\documentclass{article}

\usepackage{graphicx}
\usepackage{amsmath}
\usepackage{amssymb}
\usepackage{subcaption}
\usepackage{algorithm} 
\usepackage{algpseudocode} 
\usepackage{hyperref}
\usepackage{amsthm}
\usepackage{authblk}

\newtheorem{theorem}{Theorem}
\newtheorem{corollary}{Corollary}

\usepackage[a4paper, left=3cm, right=3cm, top=3cm]{geometry}

\begin{document}

\newcommand{\R}{\mathcal{R}}
\newcommand{\Lo}{\mathcal{L}}
\newcommand{\lopt}{\hat{\lambda}}
\newcommand{\kopt}{\hat{k}}
\newcommand{\ropt}{\hat{\rho}}
\newcommand{\nopt}{\hat{\nu}}

\newcommand{\sumk}{\sum_{k = 0}^{\kopt\left(\lopt\right)}}
\newcommand{\halb}{\frac{1}{2}}

\title{Optimal Restart Strategies for Parameter-dependent Optimization Algorithms}
\author{Lisa Schönenberger and Hans-Georg Beyer}
\affil{\small Vorarlberg University of Applied Sciences, Research Center Business Informatics, 6850, Dornbirn, Austria \\
\{lisa.schoenenberger, hans-georg.beyer\}@fhv.at}
\date{}


\maketitle

\begin{abstract}
This paper examines restart strategies for algorithms whose successful termination depends on an unknown parameter \(\lambda\). After each restart, \(\lambda\) is increased, until the algorithm terminates successfully. It is assumed that there is a unique, unknown, optimal value for \(\lambda\). For the algorithm to run successfully, this value must be reached or surpassed. The key question is whether there exists an optimal strategy for selecting \(\lambda\) after each restart taking into account that the computational costs (runtime) increases with \(\lambda\). In this work, potential restart strategies are classified into parameter-dependent strategy types. A loss function is introduced to quantify the wasted computational cost relative to the optimal strategy. A crucial requirement for any efficient restart strategy is that its loss, relative to the optimal \(\lambda\), remains bounded. To this end, upper and lower bounds of the loss are derived. Using these bounds it will be shown that not all strategy types are bounded. However, for a particular strategy type, where \(\lambda\) is increased multiplicatively by a constant factor \(\rho\), the relative loss function is bounded. Furthermore, it will be demonstrated that within this strategy type, there exists an optimal value for \(\rho\) that minimizes the maximum relative loss. In the asymptotic limit, this optimal choice of \(\rho\) does not depend on the unknown optimal \(\lambda\).
\end{abstract}

\section{Introduction}
\label{sec:int}

Optimization algorithms are often confronted with various challenges, such as becoming trapped in a local optimum or very long runtimes. Restart strategies have proven to be an effective way to overcome these obstacles. They can significantly enhance the performance and robustness of optimization algorithms.

In restart strategies, the current search process is regularly stopped and the optimization algorithm is restarted. Often a different starting point is chosen, or an algorithmic parameter is changed. The specific implementation of restart strategies can vary considerably, ranging from simple random restarts to more sophisticated techniques that adapt to the particular algorithm and problem characteristics. In general, it is not clear what the optimal implementation of a restart strategy is. Obviously, this question cannot be answered in a general way for each type of optimization algorithm. Therefore, assumptions and constraints must be made. 

The optimal implementation of a restart strategy has already been studied for Las Vegas algorithms, which are always successful, but whose running time is a random variable. In \cite{LSZ93} it was possible to derive an optimal restart strategy under the assumption that the distribution of the running time is known. In the case where the running time is unknown \cite{LSZ93} also presents an universal restart strategy, whose performance is worse than the optimal strategy by only a logarithmic factor. This topic was also explored in \cite{L21} within the context of a continuous setting. An optimal restart strategy was derived for a restricted class of continuous probability distributions. In \cite{FKQS18} the optimal strategy from \cite{LSZ93} was applied in a (1+1)-EA (Evolutionary Algorithm). The results demonstrated that this restart strategy can identify the optimum in polynomial time, whereas the conventional (1+1)-EA requires an exponential runtime. Furthermore, it was also demonstrated that restarting the (1+1)-EA outperforms the (\(\mu\)+1)-EA.

The following investigations are limited to optimization algorithms whose success depends on a certain parameter \(\lambda\). This parameter can be, for example, the population size in Evolutionary Algorithms (EA) or particle swarm optimization, or a predefined runtime. An additional assumption is that the optimization algorithm is successful\footnote{By "successful" it is meant that the algorithm has reached the domain of the global optimizer in contrast to the convergence to a local optimum.} if \(\lambda\) exceeds some unknown threshold and is unsuccessful otherwise. This means that there is an optimal choice \(\lopt\) for this parameter. The Las Vegas algorithms mentioned above generally do not satisfy this assumption. It is possible that the algorithm will not terminate after waiting for a certain time \(t\). However, if the algorithm is repeated, it may terminate after a time shorter than \(t\). 

As for the scenario considered in this work, it is provided that if the \(\lambda\) parameter is equal to or greater than the (unknown) \(\hat{\lambda}\) the EA terminates successfully, 
i.e., it approaches the global optimizer with success probability one. However, if \(\lambda < \hat{\lambda}\) the success probability drops quickly with decreasing \(\lambda\). 
This scenario is observed in Evolution Strategies (ES) \cite{BS02} optimizing certain highly multimodal objective functions such as the Rastrigin, Bohachevsky, and Ackley to name a few. The \(\lambda\) parameter in ES refers to the offspring population size. On the basis of several multimodal test functions, it has already been shown \cite{SB24} that the \(\lambda\) interval of global convergence uncertainty is rather small. Therefore, a tractable model for a theoretical analysis of these restart strategies assumes a fixed population size
\(\lambda=\hat{\lambda}\) that must be reached by the restart algorithm in an efficient manner.

As for all EA, the population size and the generations are natural numbers. Therefore \(\lambda \in \mathbb{N}\) is assumed. The condition that an optimal \(\lopt\) exists also implies that it makes no sense to restart the algorithm with the same or a smaller value for \(\lambda\). The parameter must therefore be increased for each restart.

A restart strategy defines a start value \(\lambda_0\) and determines how \(\lambda\) is changed after each restart. Thus, the restart strategy can be defined as \(\mathcal{R} = \left(\lambda_0, \lambda_1, \lambda_2, \ldots\right)\). \(\mathcal{R^*} = \left(\lambda_0, \lambda_0 \rho, \lambda_0 \rho^2, \lambda_0 \rho^3, \ldots \right)\)\footnote{If \(\lambda \in \mathbb{N}\) is a prerequisite, the values must be rounded accordingly.} is often used for such problems where \(\rho > 1\). In \cite{AH05} this was applied to ES where the population size was increased after each restart with the above scaling rule using an restart parameter of \(\rho = 2\). It was mentioned that experiments indicate that the optimal value for the restart parameter \(\rho\) is between 2 and 3. In \cite{J02} the above restart strategy was applied to the \((1 + 1)\)-EA where the maximum number of generations was increased after each restart. Also in this study, an restart parameter of 2 was chosen. Besides \(\mathcal{R^*}\), an additive restart strategy was also considered, where the restart parameter was added after each restart, i.e., \(\mathcal{R}^+ = \left(\lambda_0, \lambda_0 + \rho, \lambda_0 + 2\rho, \lambda_0 + 3\rho, \ldots \right)\). The goal of this investigation is to find the optimal choice of the restart parameter \(\rho\) for different restart strategy types.

This paper is organized as follows: The general definition of a restart strategy is given in Section~\ref{sec:RS} and the different types of restart strategies are introduced. Section \ref{sec:loss} introduces the loss function, which indicates how much computational resources are wasted compared to the optimal strategy. Upper and lower bounds for the loss function will be derived for each strategy type. A relative loss function is introduced in Section \ref{sec:rloss}. Being based on this relative loss function the optimal choice of the restart parameters that minimizes the maximal relative loss is examined. Finally, in Section~\ref{sec:end}, a summary of the results and an outlook on future research are given.

\section{Restart Strategies} \label{sec:RS}

The restart strategy (RS) under consideration is applicable to algorithms whose success depends on  an algorithmic parameter \(\lambda \in \mathbb{N}\). The algorithm \(\textbf{A}\) is only successful if this algorithmic parameter exceeds a certain bound. Formally, this can be expressed as
\begin{align} \label{eq:cond}
	&\textbf{A}(\lambda) \quad \textnormal{is successful if} \quad \lambda \geq \lopt \nonumber
	\\
	&\textbf{A}(\lambda) \quad \textnormal{is unsuccessful if} \quad \lambda < \lopt.
\end{align}
It is customary to measure the computational effort of black-box optimization algorithms, such as Evolution Strategies, by the number \(F_E(\lambda)\) of objective function evaluations that \(\textbf{A}(\lambda)\) uses until termination. As it is the case for the majority of algorithms, it can be assumed that \(F_E(\lambda)\) increases with \(\lambda\). Consequently, the optimal choice is to execute the algorithm with \(\lopt\). In this context, \(\lopt\) is also referred to as the \emph{optimal} \(\lambda\).

Evolution Strategies in multimodal landscapes meet this requirement to a satisfactory extent. This has been demonstrated in \cite{SB24} for several multimodal test functions. In this case, the parameter \(\lambda\) represents the population size. There is an interval for \(\lambda\) where  it is possible to achieve a positive success probability of less than 1. This interval, however, is relatively small in comparison to the population size required to achieve a positive success probability. If \(\lambda\) exceeds this interval, the success rate will remain constant at one. 

To approach the optimal choice of the algorithmic parameter \(\lopt\), which is generally unknown, a restart strategy can be used. Restart strategies are defined by an unbounded sequence
\begin{align} \label{eq:RS}
	\R = \left(\lambda_0, \lambda_1, \lambda_2, \ldots\right), \quad \lambda_i \in \mathbb{N},
\end{align}
where \(\lambda_k\) represents the algorithmic parameter of the \(k\)th run. The \(k\)th run of the RS is denoted with \(\R_k\). \(\R_k\) is stopped when a local stopping criterion is fulfilled. Then, an independent algorithm \(\R_{k + 1}\) with parameter \(\lambda_{k+1}\) is executed. This process is repeated until the algorithm is successful. This raises the question of how to choose \(\lambda_k\). Because of condition~(\ref{eq:cond}), it is clear that \(\lambda\) should be increased after each restart, but the optimal amount of increase is unknown. A common choice for \(\lambda_k\) is \(\lambda_k = \lambda_0 2^k\) (see for example \cite{AH05} or \cite{J02}), however, no criterion of optimality exists up until now.

In principle there are infinitely many restart strategies. Therefore, possible restart strategies are classified into parameter-dependent groups, called \emph{strategy types}. The corresponding parameter is called \emph{restart parameter}. Different strategy types will be considered. The first type increases the  population size by a constant amount, i.e.,
	\begin{align} \label{eq:Rp}
		\R^+ &= \left(\lambda_0, \lambda_1, \lambda_2, \ldots\right) \nonumber
		\\
		\lambda_k &= \lambda_{k - 1} + \nu = \lambda_0 + k \nu, \quad k \geq 1,
	\end{align}
where \(\nu \in \mathbb{N} \setminus \{0\}\) is the restart parameter. Another way to increase \(\lambda\) is to use a multiplicative change after each restart, i.e.
\begin{align} \label{eq:Rm}
	\R^* &= \left(\lambda_0, \lambda_1, \lambda_2, \ldots\right) \nonumber
	\\
	\lambda_k &= \lceil \lambda_{k - 1} \rho \rceil  = \lceil\lceil\lceil \lambda_0 \rho \rceil \rho\rceil \ldots \rho\rceil, \quad k \geq 1.
\end{align}
Because of the assumption that \(\lambda\) increases after each restart, it can be assumed that \(\rho > 1\). For \(\R^*\)-RS, \(\lambda_k\) is determined based on the previous rounded-up values. Alternatively, one can consider
\begin{align} \label{eq:Rmmod}
	\R^\times &= \left(\lambda_0, \lambda_1, \lambda_2, \ldots\right) \nonumber
	\\
	\lambda_k &= \lceil \lambda_0 \rho^{k}\rceil.
\end{align}
The \(\R^*\)-RS and \(\R^\times\)-RS are also called \emph{multiplicative strategy types}. A third type of restart strategies obeys a power law with constant \(\alpha\) defined by
\begin{align} \label{eq:Rpow}
	\R^\# &= \left(\lambda_0, \lambda_1, \lambda_2, \ldots\right) \nonumber
	\\
	\lambda_k &= \lceil \lambda_0 (k + 1)^\alpha\rceil.
\end{align}
It is assumed \(\lambda_k > \lambda_{k-1}\), which implies that \(\alpha \geq 1\).

There is no clear indication that one strategy type is more effective than the other. Furthermore, it is also not clear how to choose the restart parameters \(\nu\), \(\rho\), or \(\alpha\) for any given strategy type. The following sections investigates how the number of function evaluations is affected by the restart parameters. The objective of these investigations is to identify the optimal restart parameter for each strategy type. Furthermore, the criteria for a suitable strategy type will be defined.

\section{The Loss Function} \label{sec:loss}

When selecting the restart parameters \(\nu\), \(\rho\), or \(\alpha\), it is important to avoid choosing values that are too small as this will result in many restarts being necessary. Conversely, if the restart parameter is set to a very large value, \(\lambda\) will also become very large after just a few restarts. This can lead to \(\lambda\) being much larger than necessary, requiring more function evaluations than necessary. Let \(\lopt\) be the minimal population size needed such that an Evolutionary Algorithm EA is successful. The loss \(\Delta F_E\) of an \(\R\)-RS is defined by\footnote{The \(\rho\) is used as a substitute to indicate the dependency of \(\Delta F_E\) on the respective restart parameter.}
\begin{align} \label{eq:loss}
	\Delta F_E\left(\lopt, \rho\right) = \sumk F_E(\lambda_k) - F_E(\lopt).
\end{align}
\(F_E(\lambda_k)\) denotes the number of function evaluations that algorithm \(\textbf{A}(\lambda_k)\) uses until termination. It holds \(F_E(\lambda_k) = \lambda_{k}g_k\), where \(g_k\) is the number of generations used in the \(k\)th run. The \(\kopt\) denotes the minimum number of restarts required to attain a \(\lambda\) larger than or equal to \(\lopt\), i.e.,
\begin{align}
	\kopt\left(\lopt\right) = \arg \min \{k \mid \lambda_k \geq \lopt\}. 
\end{align}
Assuming that each run requires the same number of generations \(g\) until termination, (\ref{eq:loss}) becomes
\begin{align} \label{eq:lossSimpleg}
	\Delta F_E\left(\lopt, \rho\right) = \left(\sumk \lambda_k - \lopt\right)g.
\end{align}
The validity of this simplification does hold approximately for Evolution Strategies on certain highly multimodal objective functions as has been shown in \cite{OB24}. It is used here as a model assumption. As a result, \(g\) can be dropped in the following considerations, thus, a reduced loss function
\begin{align} \label{eq:lossSimple}
	\Lo\left(\lopt, \rho\right) = \sumk \lambda_k - \lopt
\end{align}
will be used. (\ref{eq:lossSimple}) can be calculated numerically using Alg. \ref{alg:lossNum}, where the update of \(\lambda\) is denoted by \(r\) and depends on the specific strategy type. It holds that
\begin{align}
	r(\lambda) &= \lambda + \nu \quad  \textnormal{for} \quad \R^+
	\\
	r(\lambda) &= \lceil \lambda  \rho \rceil \quad \textnormal{for} \quad \R^* 
	\\
	r(\lambda) &= \big \lceil \lambda_0  \rho^{k} \big \rceil \quad \textnormal{for} \quad \R^{\times}
	\\
	r(\lambda) &= \big \lceil \lambda_0 (k + 1)^\alpha \big \rceil \quad \textnormal{for} \quad \R^{\#},
\end{align}
where \(k\) is the number of restarts.

\begin{algorithm} 
	\caption{Numerical Calculation of the Loss Function (\ref{eq:lossSimple})} 
	\begin{algorithmic}[1]
		\State \(\mathrm{Initialize} \left(\lambda = \lambda_0, F_E = \lambda_0, k = 0\right)\)
		\While {\(\lambda < \lopt\)}
		\State \(k = k+1\)
		\State \(\lambda = r(\lambda)\)
		\hfill\Comment{update \(\lambda\), depending on strategy type}
		\State \(F_E = F_E + \lambda\)
		\EndWhile
		\State \(\Lo = F_E - \lopt\)
	\end{algorithmic} 
	\label{alg:lossNum}
\end{algorithm}

If \(\lopt = \lambda_k\) exactly \(k\) restarts are necessary, i.e., \(\kopt(\lambda_k) = k\). If \(\lopt = \lambda_k + 1\) an additional restart is necessary and it holds that \(\kopt(\lambda_k + 1) = k + 1\). Therefore the loss function (\ref{eq:lossSimple}) jumps between \(\lambda_k\) and \(\lambda_{k + 1}\). If the number of restarts \(\kopt\) is the same, the first expression in (\ref{eq:lossSimple}) does not depend on \(\lopt\). Hence, \(\Lo\) decreases linearly for \(\lambda_{k - 1} + 1 \leq \lopt \leq \lambda_{k}\), i.e., it holds for all \(k \geq 1\) that
\begin{align} \label{eq:lossD}
	\Lo(\lopt) > \Lo(\lambda_k) \quad \textnormal{for} \quad \lambda_{k - 1} + 1 \leq \lopt < \lambda_k.
\end{align}
These observations, which are independent of the specific strategy type, lead to the typical saw tooth shape of the loss function. This is visualized for the different strategy types in Figs. \ref{fig:lossPlus}, \ref{fig:lossTimes}, \ref{fig:lossMult}, and \ref{fig:lossPow} where the loss function is represented by the gray markers.

For further investigation, this saw tooth function is difficult to handle. The following subsections derive upper and lower bounds of the loss function for each strategy type. The objective is to identify sharp bounds that can be represented explicitly as a function of \(\lopt\) and the restart parameter.

\subsection{The Loss Function for the \(\R^+\)-RS}

The loss function of the \(\R^+\)-RS is denoted as \(\Lo^+(\lopt, \nu)\) and represented by the gray markers in Fig. \ref{fig:lossPlus}.

\begin{figure}
	\centering
	\hspace{-15pt}
	\begin{subfigure}{0.35\textwidth}
		\includegraphics[width=1\linewidth]{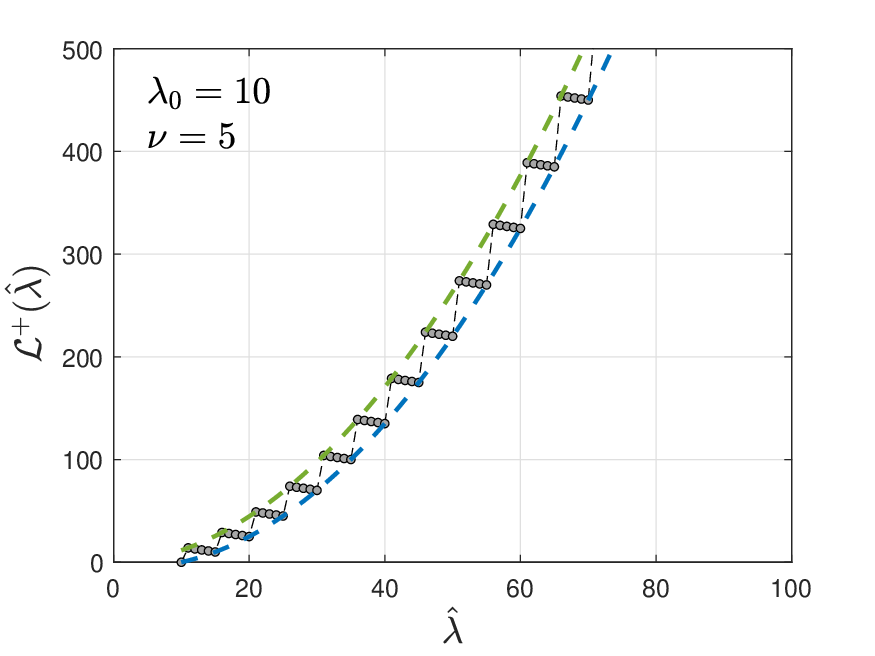}
	\end{subfigure}
	\hspace{-15pt}
	\begin{subfigure}{0.35\textwidth}
		\includegraphics[width=1\linewidth]{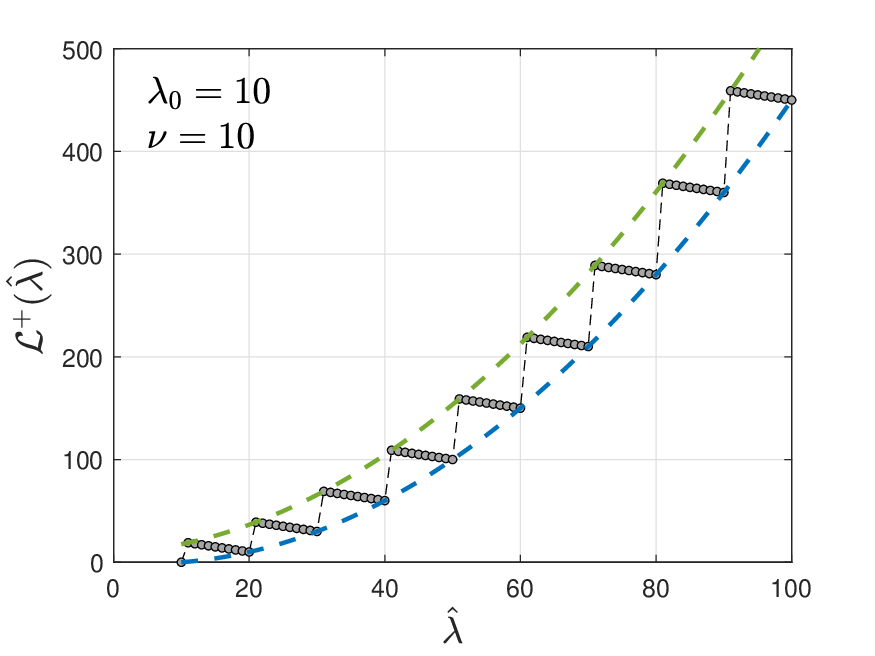}
	\end{subfigure}
	\hspace{-15pt}
	\begin{subfigure}{0.35\textwidth}
		\includegraphics[width=1\linewidth]{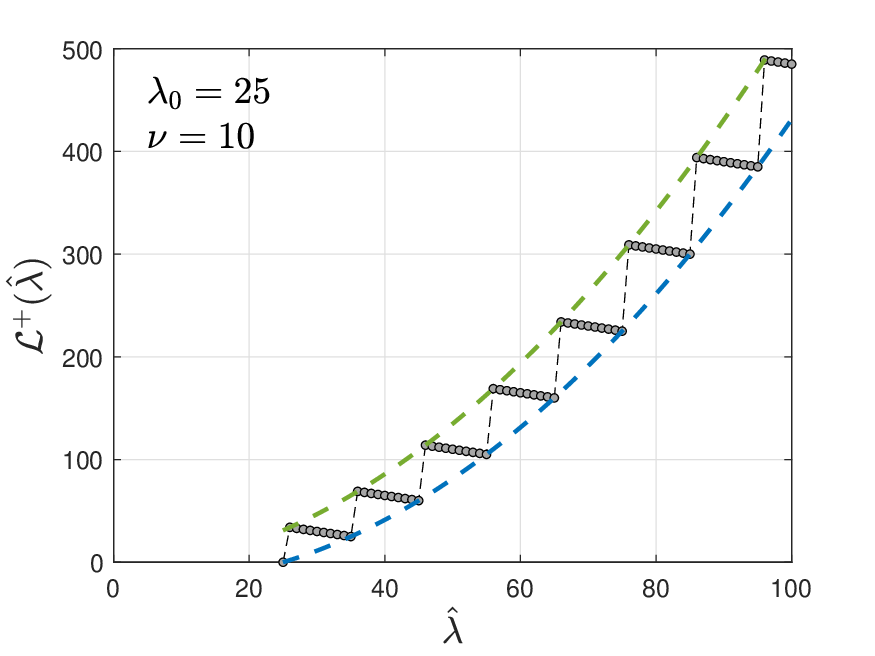}
	\end{subfigure}
	\hspace{-20pt}
	\caption{Loss function (\ref{eq:lossSimple}) for \(\R^+\) depending on \(\lopt\). Gray markers represent the numerical values of the loss function, determined with Alg. \ref{alg:lossNum}. The green dashed line shows the upper bound (\ref{eq:lossPUp}) and the blue dashed line shows the lower bound (\ref{eq:lossPLow}).}
	\label{fig:lossPlus}
\end{figure}

\begin{theorem}
	Let \(\Lo^+(\lopt, \nu)\) be the loss function (\ref{eq:lossSimple}) of the \(\R^+\)-RS with restart parameter \(\nu \in \mathbb{N} \setminus \{0\}\) and \(\lambda_k = \lambda_0 + k\nu\). Let
	\begin{align} \label{eq:lossPUp}
		\Lo^+_\mathrm{up}\left(\lopt, \nu\right) =  \halb \left(\lopt - \lambda_0 - 1\right)\left(\frac{\lopt + \lambda_0 - 1}{\nu} + 1\right) + \lambda_0 + \nu - 1.
	\end{align}
	Then \(\Lo^+_\mathrm{up}(\lopt, \nu)\) is an upper bound of \(\Lo^+(\lopt, \nu)\).
\end{theorem}

\begin{proof}
	The number of required restarts at \(\lopt = \lambda_k + 1\) is \(\kopt(\lambda_k + 1) = k + 1\). It holds that \(\lambda_k = \lambda_0 + k \nu\) and therefore
	\begin{align} 
		\Lo^+(\lambda_k + 1, \nu) &= \sum_{j = 0}^{k + 1} \lambda_j - \lambda_k - 1 = \sum_{j = 0}^{k + 1} \left(\lambda_0 + j \nu\right) -\lambda_0 - k \nu - 1 \nonumber
		\\
		&= \left(k + 2\right) \lambda_0 + \nu \sum_{j = 0}^{k + 1} j -\lambda_0 - k \nu - 1 \nonumber\\
		&= \left(k + 2\right) \lambda_0 + \frac{\nu}{2} \left(k + 2\right) \left(k + 1\right) -\lambda_0 - k \nu - 1 \nonumber
		\\
		&= \left(k + 1\right) \lambda_0 + \frac{\nu}{2} \left(k^2 + 3 k + 2\right) - k \nu- 1 \nonumber
		\\
		&= \left(k + 1\right) \lambda_0 + \frac{\nu}{2} \left(k^2 + k + 2\right) - 1 \nonumber
		\\
		&= \left(k + 1\right) \lambda_0 + \frac{\nu}{2} \left(k^2 + k\right) + \nu - 1.
	\end{align}
	Because 
	\begin{align} \label{eq:lossPlusk}
		\lambda_k = \lambda_0 + k \nu \Leftrightarrow k = \frac{\lambda_k - \lambda_0}{\nu},
	\end{align}
	it follows that for all \(k \geq 0\)
	\begin{align}
		\Lo^+(\lambda_k + 1, \nu) &= \left(\frac{\lambda_k - \lambda_0}{\nu} + 1\right) \lambda_0 + \frac{\nu}{2} \left(\left(\frac{\lambda_k - \lambda_0}{\nu}\right)^2 + \frac{\lambda_k - \lambda_0}{\nu}\right) + \nu - 1 \nonumber
		\\
		&= \frac{\lambda_0}{\nu}\left(\lambda_k - \lambda_0\right) + \lambda_0 + \frac{\nu}{2} \left(\lambda_k - \lambda_0\right) \left(\frac{\lambda_k - \lambda_0}{\nu^2} + \frac{1}{\nu}\right) + \nu - 1 \nonumber
		\\
		&= \left(\lambda_k - \lambda_0\right)\left(\frac{\lambda_0}{\nu} + \frac{\lambda_k - \lambda_0}{2\nu} + \halb\right) + \lambda_0 + \nu - 1 \nonumber
		\\
		&= \left(\lambda_k - \lambda_0\right)\left(\frac{2 \lambda_0 + \lambda_k - \lambda_0}{2\nu} + \halb\right) + \lambda_0 + \nu - 1 \nonumber
		\\ 
		&= \halb \left(\lambda_k - \lambda_0\right)\left(\frac{\lambda_k + \lambda_0}{\nu} + 1\right) + \lambda_0 + \nu - 1 \nonumber
		\\
		& = \halb \left(\lambda_k + 1 - \lambda_0 - 1\right)\left(\frac{\lambda_k + 1 + \lambda_0 - 1}{\nu} + 1\right) + \lambda_0 + \nu - 1\nonumber
		\\
		& = \Lo^+_\mathrm{up}(\lambda_k + 1, \nu).
	\end{align}
	For \(\lopt = \lambda_0\) it holds that
	\begin{align}
		\Lo^+_\mathrm{up}(\lambda_0, \nu) &= -\halb \left(\frac{2 \lambda_0 - 1}{\nu} + 1\right) + \lambda_0 + \nu - 1 \nonumber
		\\
		&= \lambda_0 \left(1 - \frac{1}{\nu}\right) + \nu + \frac{1}{2 \nu} - \frac{3}{2} \nonumber
		\\
		&\geq  \nu + \frac{1}{2 \nu} - \frac{3}{2} \geq 0 = \Lo^+(\lambda_0, \nu)
	\end{align}
	for \(\nu \geq 1\). In (\ref{eq:lossD}) it was shown that \(\Lo^+(\lopt, \nu)\) decreases between \(\lambda_{k} + 1\) and \(\lambda_{k + 1}\). Because \(\Lo^+_\mathrm{up}\) is an increasing function of \(\lopt\) it holds for all \(\lopt \geq \lambda_{0}\) that
	\begin{align}
		\Lo^+_\mathrm{up}(\lopt, \nu) \geq \Lo^+(\lopt, \nu).
	\end{align}
\end{proof}
\noindent
The upper bound (\ref{eq:lossPUp}) is represented by the green dashed line in Fig. \ref{fig:lossPlus}.

\begin{theorem} \label{th:pLow}
	Let \(\Lo^+(\lopt, \nu)\) be the loss function (\ref{eq:lossSimple}) of the \(\R^+\)-RS with restart parameter \(\nu \in \mathbb{N} \setminus \{0\}\) and \(\lambda_k = \lambda_0 + k\nu\). Let
	\begin{align} \label{eq:lossPLow}
		\Lo^+_\mathrm{low}\left(\lopt, \nu\right) :=  \halb \left(\lopt - \lambda_0\right)\left(\frac{\lopt + \lambda_0}{\nu} - 1\right).
	\end{align}
	Then \(\Lo^+_\mathrm{low}(\lopt, \nu)\) is a lower bound of \(\Lo^+(\lopt, \nu)\).
\end{theorem}

\begin{proof}
	The number of required restarts at \(\lopt = \lambda_k\) is \(\kopt(\lambda_k) = k\). It follows from (\ref{eq:lossPlusk}) that
	\begin{align} \label{eq:lossP0}
		\Lo^+(\lambda_k, \nu) &= \sum_{j = 0}^{k} \lambda_j - \lambda_k = \sum_{j = 0}^{k - 1} \lambda_j \nonumber
		\\
		&= \sum_{j = 0}^{k - 1} \left(\lambda_0 + j \nu\right) = k \lambda_0 + \nu \sum_{j = 0}^{k - 1} j \nonumber
		\\
		&= k \lambda_0 + \frac{\nu}{2} k \left(k - 1\right) \nonumber
		\\
		&= \frac{\lambda_k - \lambda_0}{\nu} \lambda_0 +  \frac{\nu}{2} \left(\frac{\lambda_k - \lambda_0}{\nu}\right) \left(\frac{\lambda_k - \lambda_0}{\nu} - 1\right) \nonumber \\
		&= \frac{\lambda_k - \lambda_0}{\nu} \lambda_0 + \frac{\left(\lambda_k - \lambda_0\right)^2}{2\nu} - \frac{\lambda_k - \lambda_0}{2} \nonumber
		\\
		&= \left(\lambda_k - \lambda_0\right) \left(\frac{\lambda_0}{\nu} +\frac{\lambda_k - \lambda_0}{2\nu} - \halb\right) \nonumber
		\\ 
		&= \left(\lambda_k - \lambda_0\right) \left(\frac{2 \lambda_0 + \lambda_k - \lambda_0}{2\nu}  - \halb\right) \nonumber
		\\
		&= \halb \left(\lambda_k - \lambda_0\right)\left(\frac{\lambda_k + \lambda_0}{\nu} - 1\right) \nonumber
		\\
		&= \Lo^+_\mathrm{low}(\lambda_k, \nu),
	\end{align}
	for all \(k \geq 0\). In (\ref{eq:lossD}) it was shown that \(\Lo^+(\lopt, \nu)\) decreases between \(\lambda_{k - 1} + 1\) and \(\lambda_{k}\). Therefore, it follows for \(\lambda_{k - 1}  + 1 \leq \lopt < \lambda_{k}\) by using (\ref{eq:lossP0}) that
	\begin{align}
		\Lo^+(\lopt, \nu) > \Lo^+(\lambda_k, \nu) &= \Lo^+_\mathrm{low}(\lambda_k, \nu) = \halb \left(\lambda_k - \lambda_0\right)\left(\frac{\lambda_k + \lambda_0}{\nu} - 1\right) \nonumber\\
		&> \halb \left(\lopt - \lambda_0\right)\left(\frac{\lopt + \lambda_0}{\nu} - 1\right)  = \Lo^+_{\mathrm{low}}(\lopt, \nu),
	\end{align}
	which holds for all \(k \geq 1\). Because (\ref{eq:lossP0}) includes the case \(\lambda_0\), it follows for all \(\lopt \geq \lambda_0\) that
	\begin{align}
		\Lo^+(\lopt, \nu) \geq \Lo^+_\mathrm{low}(\lopt, \nu).
	\end{align}
\end{proof}
\noindent
The lower bound (\ref{eq:lossPLow}) is represented by the blue dashed line in Fig. \ref{fig:lossPlus}.

\subsection{The Loss Function for the \(\R^{\times}\)-RS}
The loss function of the \(\R^\times\)-RS is denoted as \(\Lo^\times(\lopt, \rho)\) and represented by the gray markers in Fig.~\ref{fig:lossTimes}.

\begin{figure}
	\centering
	\hspace{-15pt}
	\begin{subfigure}{0.35\textwidth}
		\includegraphics[width=1\linewidth]{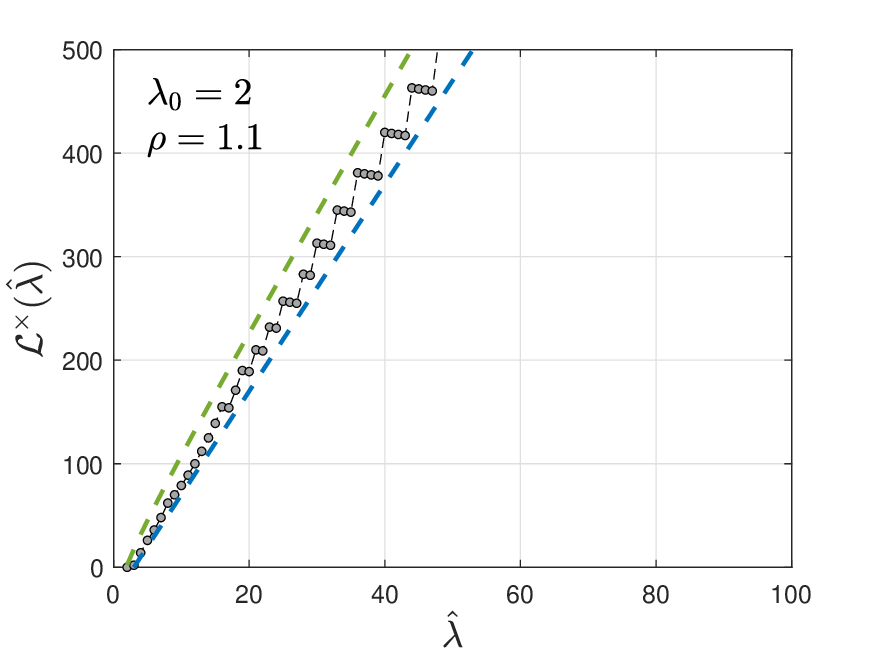}
	\end{subfigure}
	\hspace{-15pt}
	\begin{subfigure}{0.35\textwidth}
		\includegraphics[width=1\linewidth]{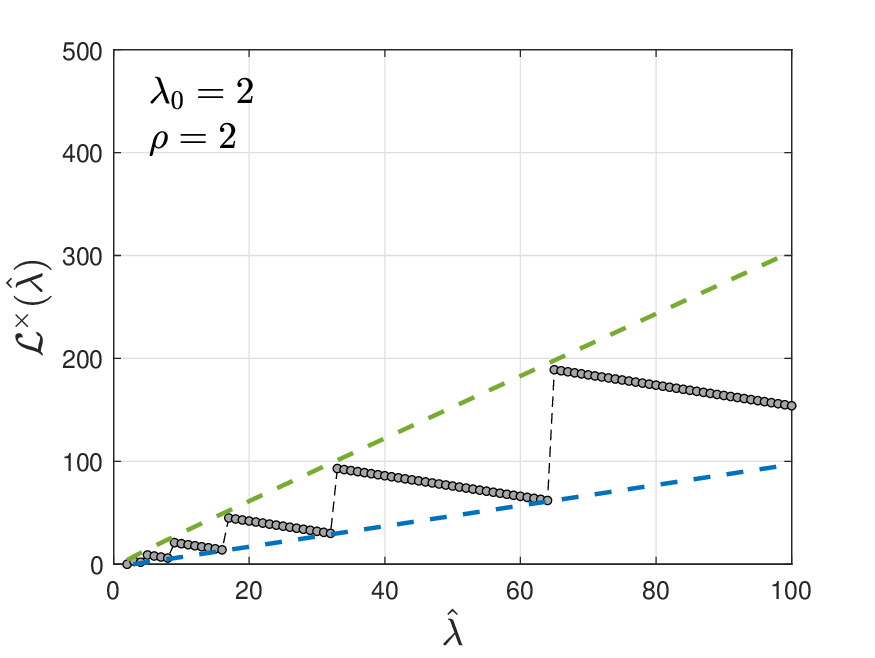}
	\end{subfigure}
	\hspace{-15pt}
	\begin{subfigure}{0.35\textwidth}
		\includegraphics[width=1\linewidth]{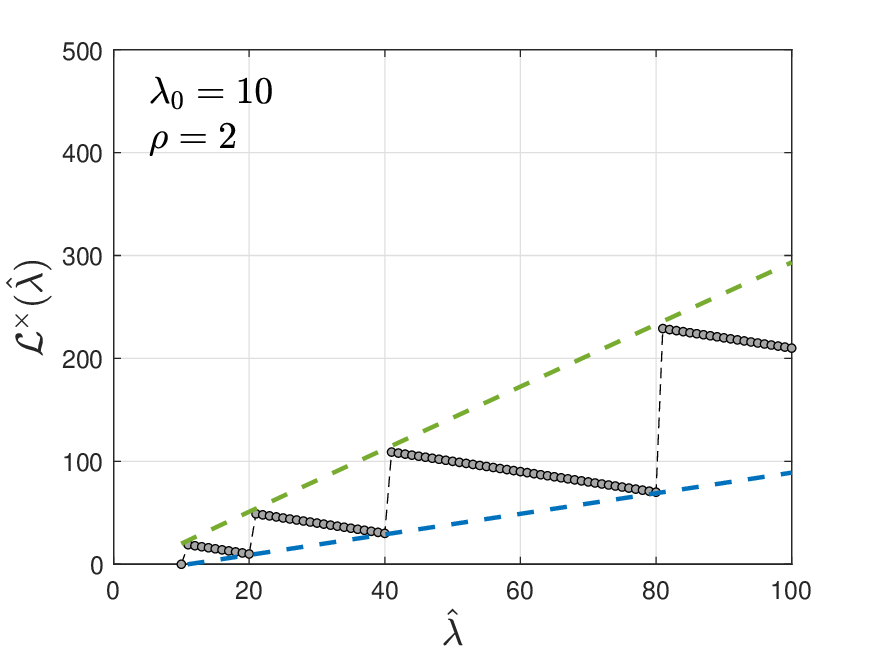}
	\end{subfigure}
	\hspace{-20pt}
	\caption{Loss function (\ref{eq:lossSimple}) for \(\R^\times\) depending on \(\lopt\). Gray markers represent the numerical values of the loss function, determined with Alg. \ref{alg:lossNum}. The green dashed line shows the upper bound (\ref{eq:lossModUp}) and the blue dashed line shows the lower bound (\ref{eq:lossModLow}).}
	\label{fig:lossTimes}
\end{figure}

\begin{theorem} 
	Let	\(\Lo^\times(\lopt, \rho)\) be the loss function (\ref{eq:lossSimple}) of the \(\R^\times\)-RS with restart parameter \(\rho > 1\) and \(\lambda_{k} = \lceil \lambda_0 \rho^k \rceil\). Let
	\begin{align} \label{eq:lossModUp}
		\Lo^{\times}_\mathrm{up}\left(\lopt, \rho\right) :=  \lopt \left(\rho + \frac{1}{\rho - 1}\right) - \frac{\lambda_0}{\rho - 1} + \frac{\ln\left(\lopt / \lambda_0\right)}{\ln\left(\rho\right)}.
	\end{align}
	Then \(\Lo^{\times}_\mathrm{up}(\lopt, \rho)\) is an upper bound of \(\Lo^\times(\lopt, \rho)\).
\end{theorem}
\begin{proof}
	The number of required restarts at \(\lopt = \lambda_{k} + 1\) is \(\kopt(\lambda_k + 1) = k + 1\). It holds that
	\begin{align} \label{eq:lambdaMod}
		\lambda_0 \rho^k \leq \lceil \lambda_0 \rho^k \rceil = \lambda_k < \lambda_0 \rho^k + 1.
	\end{align}
	Therefore,
	\begin{align} \label{eq:loss0}
		\Lo^\times(\lambda_k + 1,  \rho) &= \sum_{j = 0}^{k + 1} \lambda_j - \lambda_k - 1 = \sum_{j = 0}^{k - 1} \lambda_j + \lambda_{k + 1} - 1 \nonumber
		\\
		&< \sum_{j = 0}^{k - 1} \left(\lambda_0\rho^j + 1\right) + \lambda_0 \rho^{k + 1} + 1 - 1 = \lambda_0 \sum_{j = 0}^{k - 1} \rho^j + k +\lambda_0 \rho^{k + 1} \nonumber
		\\
		&= \lambda_0 \frac{\rho^k - 1}{\rho - 1} + k + \lambda_0 \rho^{k + 1}.
	\end{align}
	It follows from (\ref{eq:lambdaMod}) that
	\begin{align}
		\lambda_0 \rho^k < \lambda_k + 1 \Leftrightarrow \rho^k < \frac{\lambda_k + 1}{\lambda_0} \Leftrightarrow k < \frac{\ln\left(\left(\lambda_k + 1\right)/ \lambda_0\right)}{\ln\left(\rho\right)}.
	\end{align}
	Inserting this into (\ref{eq:loss0}), then it follows for all \(k \geq 0\) that
	\begin{align}
		\Lo^\times(\lambda_k + 1, \rho) &< \lambda_0 \frac{\frac{\lambda_k + 1}{\lambda_0} - 1}{\rho - 1} + \frac{\ln\left(\left(\lambda_k + 1\right)  / \lambda_0\right)}{\ln\left(\rho\right)} + \lambda_0 \frac{\lambda_k + 1}{\lambda_0} \rho \nonumber
		\\
		&= \frac{\lambda_k + 1}{\rho - 1} - \frac{\lambda_0}{\rho - 1} + \left(\lambda_k + 1\right) \rho + \frac{\ln\left(\left(\lambda_k + 1\right)  / \lambda_0\right)}{\ln\left(\rho\right)} \nonumber
		\\
		&= \left(\lambda_k + 1\right) \left(\rho + \frac{1}{\rho - 1}\right) - \frac{\lambda_0}{\rho - 1} + \frac{\ln\left(\left(\lambda_k + 1\right)  / \lambda_0\right)}{\ln\left(\rho\right)} = \Lo^\times_\mathrm{up}(\lambda_k + 1, \rho).
	\end{align}
	For \(\lopt = \lambda_0\) it holds that
	\begin{align}
		\Lo^\times(\lambda_0, \rho) = 0 < \lambda_0 \rho = \Lo^\times_\mathrm{up}(\lopt, \rho).
	\end{align}
	In (\ref{eq:lossD}) it was shown that \(\Lo^\times(\lopt, \rho)\) decreases between \(\lambda_{k} + 1\) and \(\lambda_{k + 1}\). Because \(\left(\rho + \frac{1}{\rho - 1}\right) > 0\) for \(\rho > 1\) it holds that \(\Lo^\times_\mathrm{up}\) increases with \(\lopt\). Therefore,
	\begin{align}
		\Lo^\times(\lopt, \rho) < \Lo^\times_\mathrm{up}(\lopt, \rho)
	\end{align}
	for all \(\lopt \geq \lambda_{0}\).
\end{proof}
\noindent
The upper bound (\ref{eq:lossModUp}) is represented by the green dashed line in Fig. \ref{fig:lossTimes}. 

\begin{theorem} 
	Let \(\Lo^\times(\lopt, \rho)\) be the loss function (\ref{eq:lossSimple}) of the \(\R^\times\)-RS with restart parameter \(\rho > 1\) and \(\lambda_{k} = \lceil \lambda_0 \rho^k \rceil\). Let
	\begin{align} \label{eq:lossModLow}
		\Lo^{\times}_\mathrm{low}\left(\lopt, \rho\right) := \frac{\lopt - 1 - \lambda_0}{\rho - 1}.
	\end{align}
	Then \(\Lo^{\times}_\mathrm{low}(\lopt, \rho)\)	is a lower bound of \(\Lo^\times(\lopt, \rho)\).
\end{theorem}
\begin{proof}
	The number of required restarts at \(\lopt = \lambda_{k}\) is \(\kopt(\lambda_k ) = k\). Using \(\lambda_k \geq \lambda_0 \rho^k\) from (\ref{eq:lambdaMod}), then
	\begin{align} \label{eq:loss1}
		\Lo^\times(\lambda_k, \rho) = \sum_{j = 0}^{k} \lambda_j - \lambda_k = \sum_{j = 0}^{k - 1} \lambda_j \geq \sum_{j = 0}^{k - 1} \lambda_0\rho^j = \lambda_0 \frac{\rho^{k} - 1}{\rho - 1}.
	\end{align}
	With (\ref{eq:lambdaMod}) it holds that
	\begin{align}
		\lambda_0 \rho^k + 1 > \lambda_k \Leftrightarrow \rho^k > \frac{\lambda_{k} - 1}{\lambda_0}.
	\end{align}
	Inserting this into (\ref{eq:loss1}), then it follows that
	\begin{align} \label{eq:loss2}
		\Lo^\times(\lambda_k, \rho) &> \lambda_0 \frac{\frac{\lambda_k - 1}{\lambda_0} - 1}{\rho - 1} = \frac{\lambda_k - 1 - \lambda_0}{\rho - 1}.
	\end{align}
for all \(k \geq 0\). In (\ref{eq:lossD}) it was shown that \(\Lo^\times(\lopt, \rho)\) decreases between \(\lambda_{k - 1} + 1\) and \(\lambda_{k}\). Therefore, it follows for \(\lambda_{k - 1}  + 1 \leq \lopt < \lambda_{k}\) by using (\ref{eq:loss2})
	\begin{align}
		\Lo^\times(\lopt, \rho) > \Lo^\times\left(\lambda_{k}, \rho \right)> \frac{\lambda_{k} - 1 - \lambda_0}{\rho - 1} > \frac{\lopt - 1 - \lambda_0}{\rho - 1} =  \Lo^\times_\mathrm{low}(\lopt, \rho),
	\end{align}
	which holds for all \(k \geq 1\). Because (\ref{eq:loss2}) includes the case \(\lambda_0\), it follows for all \(\lopt \geq \lambda_0\) that
	\begin{align}
		\Lo^\times(\lopt, \rho) > \Lo^\times_\mathrm{low}(\lopt, \rho).
	\end{align}
\end{proof}
\noindent
The lower bound (\ref{eq:lossModLow}) is represented by the blue dashed line in Fig. \ref{fig:lossTimes}.

\subsection{The Loss Function for the \(\R^*\)-RS} \label{sec:lossMult}

The loss function of the \(\R^*\)-RS is denoted as \(\Lo^*(\lopt, \rho)\) and represented by the gray markers in Fig.~\ref{fig:lossMult}. In order to derive an upper bound for \(\R^*\), the following theorem is needed.

\begin{figure}
	\centering
	\hspace{-15pt}
	\begin{subfigure}{0.35\textwidth}
		\includegraphics[width=1\linewidth]{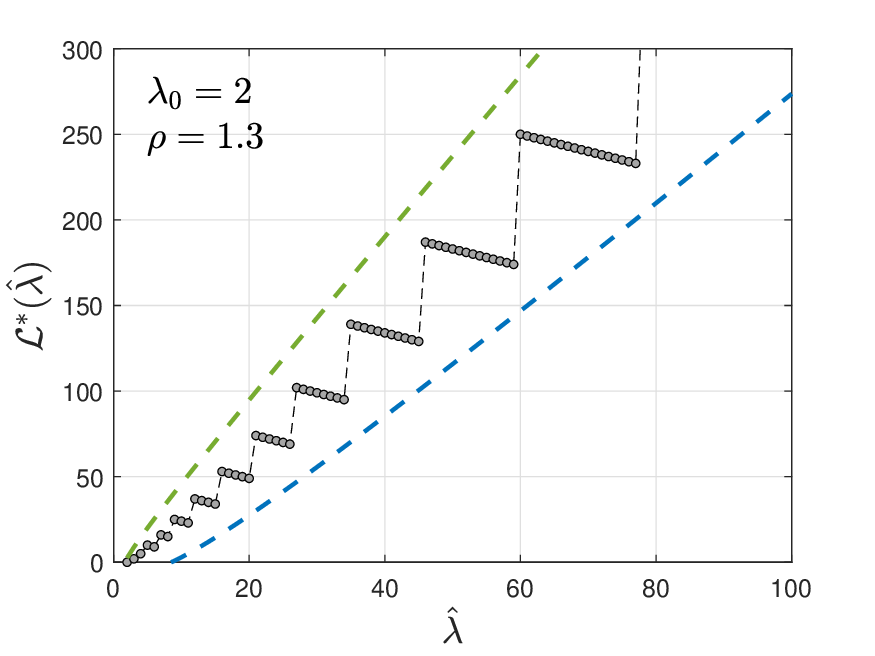}
	\end{subfigure}
	\hspace{-15pt}
	\begin{subfigure}{0.35\textwidth}
		\includegraphics[width=1\linewidth]{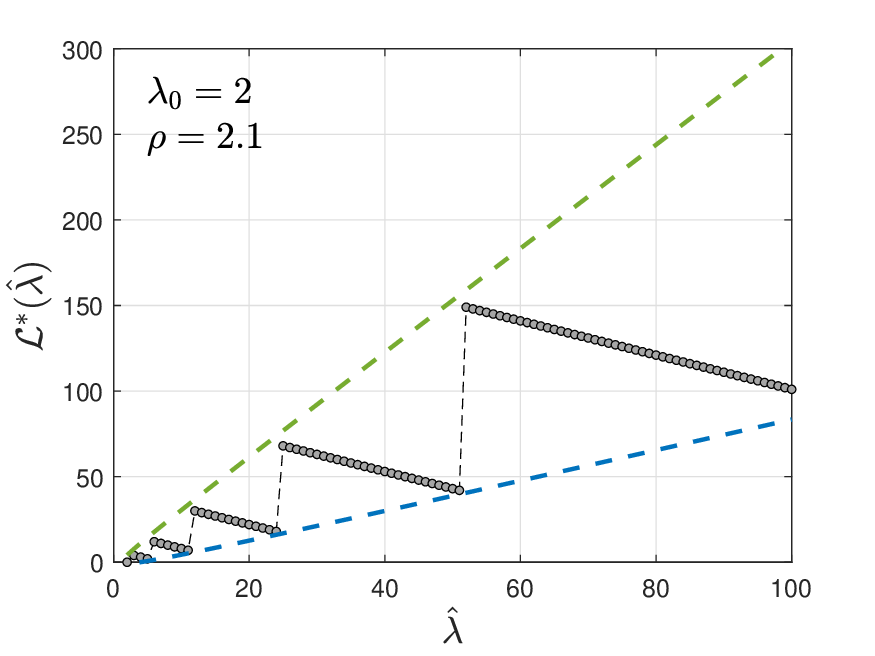}
	\end{subfigure}
	\hspace{-15pt}
	\begin{subfigure}{0.35\textwidth}
		\includegraphics[width=1\linewidth]{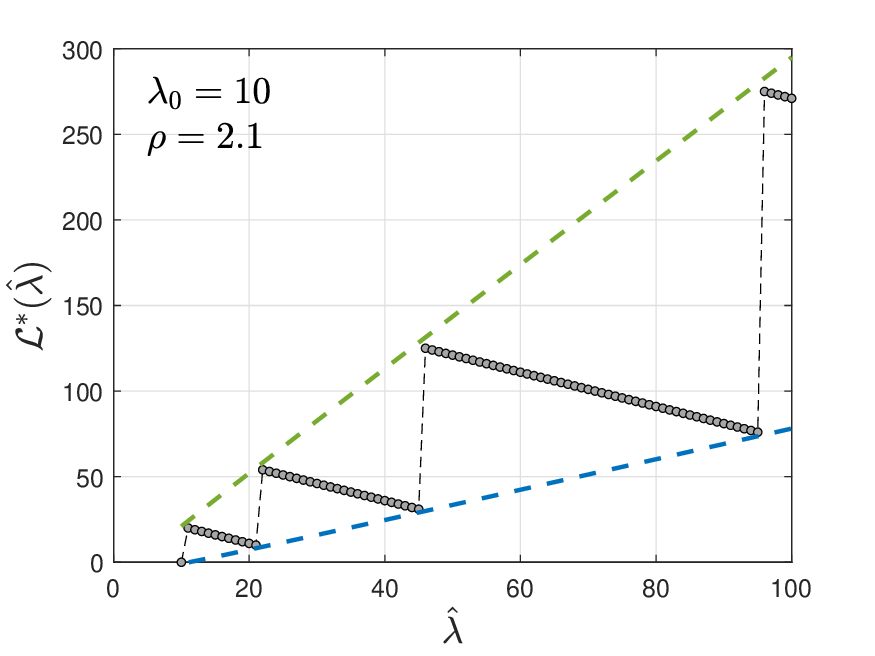}
	\end{subfigure}
	\hspace{-20pt}
	\caption{Loss function (\ref{eq:lossSimple}) for \(\R^*\) depending on \(\lopt\). Gray markers represent the numerical values of the loss function, determined with Alg. \ref{alg:lossNum}. The green dashed line shows the upper bound (\ref{eq:lossUp}) and the blue dashed line shows the lower bound (\ref{eq:lLow}).}
	\label{fig:lossMult}
\end{figure}

\begin{theorem} \label{th:m1} 
	For \(\rho > 1\), \(k \in \mathbb{N}\) and \(\lambda_{k + 1} = \lceil \lambda_{k} \rho \rceil\) let
	\begin{align}
		\mathcal{F}_\mathrm{up}(k) := \lambda_0 \rho + k + \left(\lambda_k - \lambda_0\right) \left(\rho + \frac{1}{\rho - 1}\right),
	\end{align}
	and
	\begin{align}
		\mathcal{F}(k) := \sum_{j = 0}^{k + 1} \lambda_{j} - \lambda_{k} - 1.
	\end{align}
	Then it holds that \(\mathcal{F}_\mathrm{up}(k) > \mathcal{F}(k)\).
\end{theorem}

\begin{proof} It holds for all \(k \geq 0\) that
	\begin{align} \label{eq:m0}
		\lambda_k \rho + 1 > \lambda_{k + 1} = \lceil \lambda_{k} \rho \rceil \geq \lambda_{k}\rho > \lceil \lambda_{k} \rho \rceil - 1.
	\end{align}
	It follows for \(k = 0\) that
	\begin{align}
		\mathcal{F}_\mathrm{up}(0) = \lambda_0 \rho > \lceil \lambda_{0} \rho \rceil - 1 = \lambda_1 - 1 = \sum_{j = 0}^{1} \lambda_{j} - \lambda_{0} - 1 = \mathcal{F}(0).
	\end{align}
	Assume that the condition \(\mathcal{F}_\mathrm{up}(k) > \mathcal{F}(k)\) holds for \(k\), then it follows by induction
	\begin{align}
		\mathcal{F}_\mathrm{up}(k+1) &= \lambda_0 \rho + k + 1 + \left(\lambda_{k + 1} - \lambda_0\right) \left(\rho + \frac{1}{\rho - 1}\right) \nonumber
		\\
		&= \lambda_0 \rho + k + \left(\lambda_{k} - \lambda_0\right) \left(\rho + \frac{1}{\rho - 1}\right) - \lambda_{k} \left(\rho + \frac{1}{\rho - 1}\right) + \lambda_{k + 1} \left(\rho + \frac{1}{\rho - 1}\right) + 1 \nonumber
		\\
		&= \mathcal{F}_\mathrm{up}(k) + \left(\lambda_{k + 1} - \lambda_{k}\right) \left(\rho + \frac{1}{\rho - 1}\right) + 1\nonumber
		\\
		&> \mathcal{F}(k) + \left(\lambda_{k + 1} - \lambda_{k}\right) \left(\rho + \frac{1}{\rho - 1}\right) + 1 \nonumber\\
		&= \sum_{j = 0}^{k + 1} \lambda_{j} - \lambda_{k} - 1 + \left(\lambda_{k + 1} - \lambda_{k}\right) \left(\rho + \frac{1}{\rho - 1}\right) + 1 \nonumber
		\\
		&= \sum_{j = 0}^{k + 2} \lambda_{j} - \lambda_{k + 1} - 1 - \lambda_{k + 2} + \lambda_{k + 1} + 1 - \lambda_{k} - 1 + \left(\lambda_{k + 1} - \lambda_{k}\right) \left(\rho + \frac{1}{\rho - 1}\right) + 1 \nonumber
		\\
		& = \mathcal{F}(k + 1) - \lambda_{k + 2} + \lambda_{k + 1} - \lambda_{k} + \left(\lambda_{k + 1} - \lambda_{k}\right) \left(\rho + \frac{1}{\rho - 1}\right) + 1 \nonumber
		\\
		& = \mathcal{F}(k + 1) - \lambda_{k + 2} + \lambda_{k + 1} - \lambda_{k} + \lambda_{k + 1} \rho - \lambda_{k} \rho + \frac{\lambda_{k + 1} - \lambda_k}{\rho - 1} + 1 \nonumber
		\\
		& > \mathcal{F}(k + 1) - \lambda_{k + 1} \rho - 1 + \lambda_{k + 1} - \lambda_{k} + \lambda_{k + 1} \rho - \lambda_{k} \rho + \frac{\lambda_{k + 1} - \lambda_k}{\rho - 1}+ 1 \nonumber
		\\
		&\geq \mathcal{F}(k + 1) + \lambda_{k} \rho - \lambda_{k} - \lambda_{k} \rho + \frac{\lambda_{k} \rho - \lambda_{k}}{\rho - 1} \nonumber 
		\\
		&= \mathcal{F}(k + 1),
	\end{align}
	where (\ref{eq:m0}) was used for the last two inequalities.
\end{proof}
\noindent
Using this result it is possible to derive an upper bound for the loss function of the \(\R^*\)-RS.

\begin{corollary}	
	Let \(\Lo^*(\lopt, \rho)\) be the loss function (\ref{eq:lossSimple}) of the \(\R^*\)-RS with restart parameter \(\rho > 1\) and \(\lambda_{k} = \lceil \lambda_{k - 1} \rho \rceil\) for \(k \geq 1\). Let
	\begin{align} \label{eq:lossUp}
		\Lo^*_\mathrm{up}(\lopt, \rho) := \lambda_0 \rho + \frac{\ln\left(\lopt / \lambda_0\right)}{\ln\left(\rho\right)} + \left(\lopt - \lambda_0\right) \left(\rho + \frac{1}{\rho - 1}\right).
	\end{align}
	Then \(\Lo^{*}_\mathrm{up}(\lopt, \rho)\) is an upper bound of \(\Lo^*(\lopt, \rho)\).
\end{corollary}

\begin{proof} Because \(\lambda_k = \lceil \lambda_{k - 1} \rho \rceil = \lceil \lceil\lambda_0 \rho \rceil \ldots \rho \rceil \geq \lambda_0 \rho^k\) it follows that
	\begin{align} \label{eq:mk}
		\frac{\lambda_k}{\lambda_0} \geq \rho^k \Leftrightarrow \frac{\ln\left(\lambda_k / \lambda_0\right)}{\ln\left(\rho\right)} \geq k.
	\end{align}
	By using (\ref{eq:mk}) and Theorem \ref{th:m1} it follows for \(\lambda_{k} + 1 \leq \lopt \leq \lambda_{k + 1}\) and for all \(k \geq 0\) that
	\begin{align} \label{eq:m2}
		\Lo^*_\mathrm{up}(\lopt, \rho) &= \lambda_0 \rho + \frac{\ln\left(\lopt / \lambda_0\right)}{\ln\left(\rho\right)} + \left(\lopt - \lambda_0\right) \left(\rho + \frac{1}{\rho - 1}\right) \nonumber
		\\
		&> \lambda_0 \rho + \frac{\ln\left(\lambda_k / \lambda_0\right)}{\ln\left(\rho\right)} + \left(\lambda_k - \lambda_0\right) \left(\rho + \frac{1}{\rho - 1}\right) \nonumber
		\\
		&\geq \lambda_0 \rho + k + \left(\lambda_k - \lambda_0 \right) \left(\rho + \frac{1}{\rho - 1}\right) = \mathcal{F}_\mathrm{up}(k) \nonumber
		\\
		&> \mathcal{F}(k) = \Lo^*(\lambda_{k} + 1, \rho) \geq \Lo^*(\lopt, \rho).
	\end{align}
The last inequality follows from (\ref{eq:lossD}), i.e., \(\Lo^*(\lopt, \rho)\) decreases between \(\lambda_{k} + 1\) and \(\lambda_{k + 1}\). Because (\ref{eq:m2}) holds for all \(k \geq 0\) and
	\begin{align}
		\Lo^*_\mathrm{up}(\lambda_0, \rho) = \lambda_0 \rho > 0 = \Lo^*(\lambda_0, \rho)
	\end{align}
	it follows for all \(\lopt \geq \lambda_0\) that
	\begin{align}
		\Lo^*_\mathrm{up}(\lopt, \rho) > \Lo^*(\lopt, \rho).
	\end{align}
\end{proof}
The upper bound (\ref{eq:lossUp}) is represented in Fig. \ref{fig:lossMult} by the green dashed line. In the left figure, where \(\lambda_0\) and \(\rho\) are small, the differences between (\ref{eq:lossUp}) and the upper corners of the loss function are clearly visible. In the right figure for larger values of \(\lambda_0\) and \(\rho\) the upper bound (\ref{eq:lossUp}) is only slightly larger than the upper corners of the loss function.

In Fig. \ref{fig:lossComp}, the \(\R^*\)-RS is compared with the \(\R^\times\)-RS. The black markers represent the loss of an \(\R^*\)-RS, while the gray markers represent the loss of an \(\R^\times\)-RS. The figure illustrates that the loss functions of the two strategy types are nested within one another. By transforming the upper bound of the \(\R^*\)-RS (\ref{eq:lossUp}), the upper bound of the \(\R^\times\)-RS (\ref{eq:lossModUp}) is obtained, i.e., it holds that \(\Lo_\mathrm{up}^*(\lopt, \rho) = \Lo_\mathrm{up}^\times(\lopt, \rho)\). The upper bound is illustrated by the green line in Fig. \ref{fig:lossComp}.

The next theorem provides a preparatory step for the calculation of a lower bound of \(\R^*\).

\begin{figure}
	\centering
	\begin{subfigure}{0.45\textwidth}
		\includegraphics[width=1\linewidth]{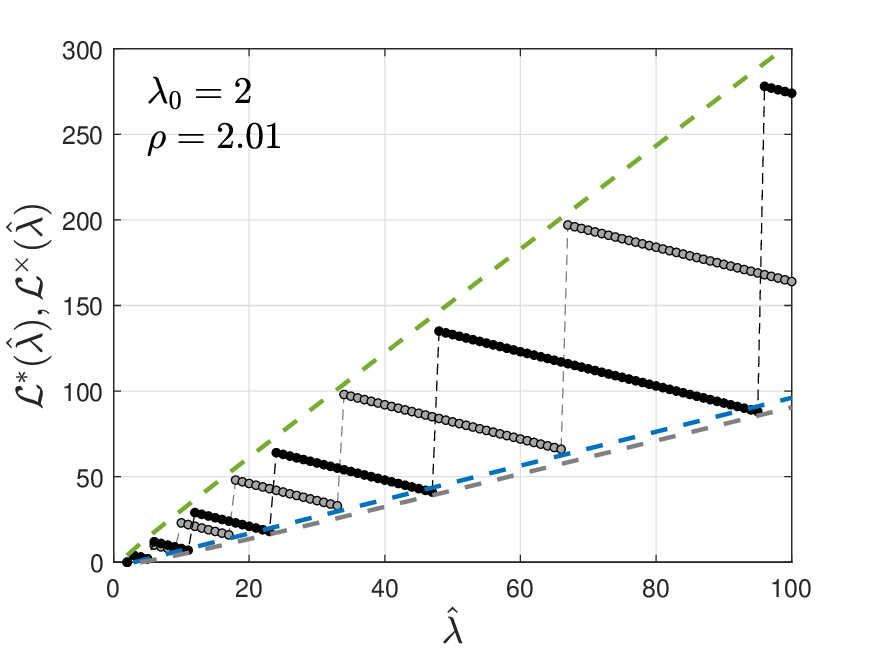}
	\end{subfigure}
	\begin{subfigure}{0.45\textwidth}
		\includegraphics[width=1\linewidth]{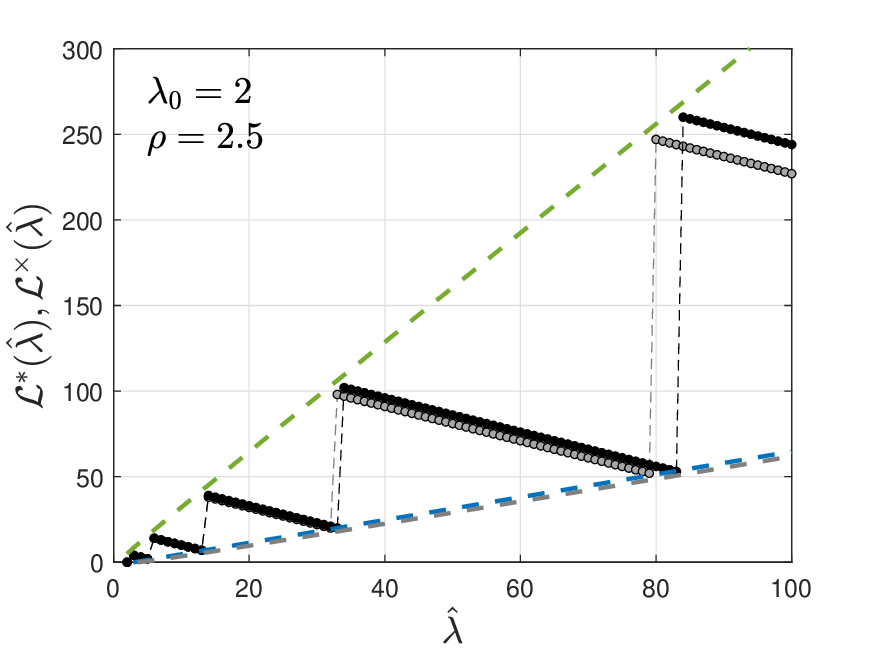}
	\end{subfigure}
	\caption{Loss function (\ref{eq:lossSimple}) depending on \(\lopt\). Markers represent the numerical values of the loss function, determined with Alg. \ref{alg:lossNum}, black for the \(\R^*\)-RS and gray for the \(R^\times\)-RS. The green line in the middle and right plot shows the upper bound (\ref{eq:lossUp}). The gray dashed line shows the lower bound (\ref{eq:lLow}) of the \(\R^*\)-RS and the blue dashed line shows the lower bound (\ref{eq:lossModLow}) of the \(\R^\times\)-RS.}
	\label{fig:lossComp}
\end{figure}

\begin{theorem} \label{th:m2}
	For \(\rho > 1\), \(k \in \mathbb{N}\setminus\{0\}\) and \(\lambda_{k + 1} = \lceil \lambda_{k} \rho \rceil\) let
	\begin{align}
		\mathcal{F}_\mathrm{low}(k) := \frac{1}{\rho - 1} \left(\lambda_k - \lambda_0 - k\right),
	\end{align}
	and
	\begin{align}
		\mathcal{F}(k) := \sum_{j = 0}^{k} \lambda_{j} - \lambda_{k}.
	\end{align}
	Then it holds that \(\mathcal{F}_\mathrm{low}(k) < \mathcal{F}(k).\)
\end{theorem}

\begin{proof}
	It holds for all \(k \geq 0\) that
	\begin{align} \label{eq:m1}
		\lambda_{k + 1} = \lceil \lambda_{k} \rho \rceil < \lambda_{k}\rho + 1.
	\end{align}
	It follows for \(k = 1\)
	\begin{align}
		\mathcal{F}_\mathrm{low}(1) = \frac{1}{\rho - 1} \left(\lambda_1 - \lambda_0 - 1\right) < \frac{1}{\rho - 1} \left(\lambda_0 \rho - \lambda_0\right) = \lambda_0 = \sum_{j = 0}^{1} \lambda_{j} - \lambda_{1} = \mathcal{F}(1).
	\end{align}
	Assume that the condition \(\mathcal{F}_\mathrm{low}(k) < \mathcal{F}(k)\) holds for \(k\), then it follows by induction that
	\begin{align}
		\mathcal{F}_\mathrm{low}(k + 1) &= \frac{1}{\rho - 1} \left(\lambda_{k + 1} - \lambda_0 - k - 1\right) \nonumber
		\\
		&=  \frac{1}{\rho - 1} \left(\lambda_{k} - \lambda_0 - k\right) + \frac{1}{\rho - 1} \left(\lambda_{k + 1} - \lambda_{k} - 1\right) \nonumber
		\\
		&=
		\mathcal{F}_\mathrm{low}(k) + \frac{1}{\rho - 1} \left(\lambda_{k + 1} - \lambda_{k} - 1\right) \nonumber
		\\
		&< \mathcal{F}(k) + \frac{1}{\rho - 1} \left(\lambda_{k + 1} - \lambda_{k} - 1\right) \nonumber
		\\
		&= \sum_{j = 0}^{k} \lambda_{j} - \lambda_{k} + \frac{1}{\rho - 1} \left(\lambda_{k + 1} - \lambda_{k} - 1\right)\nonumber
		\\
		&=\sum_{j = 0}^{k + 1} \lambda_{j} - \lambda_{k + 1} - \lambda_{k } + \frac{1}{\rho - 1} \left(\lambda_{k + 1} - \lambda_{k} - 1\right) \nonumber
		\\
		&=  \mathcal{F}(k + 1) - \lambda_{k} + \frac{1}{\rho - 1} \left(\lambda_{k + 1} - \lambda_{k} - 1\right) \nonumber
		\\
		&< \mathcal{F}(k + 1) - \lambda_{k} + \frac{1}{\rho - 1} \left(\lambda_{k}\rho + 1 - \lambda_{k} - 1\right) = \mathcal{F}(k + 1),
	\end{align}
	where (\ref{eq:m1}) was used for the last inequality.
\end{proof}
\noindent
Using Theorem \ref{th:m2}, the lower bound of \(\R^*\) is given below.

\begin{corollary}
	Let \(\Lo^*(\lopt, \rho)\) be the loss function (\ref{eq:lossSimple}) of the \(\R^*\)-RS with restart parameter \(\rho > 1\) and \(\lambda_{k} = \lceil \lambda_{k - 1} \rho \rceil\) for \(k \geq 1\). Let
	\begin{align} \label{eq:lLow}
		\Lo^*_\mathrm{low}(\lopt, \rho) := \frac{1}{\rho - 1} \left(\lopt - \lambda_{0} - \frac{\ln\left(\lopt / \lambda_0\right)}{\ln\left(\rho\right)} - 1\right).
	\end{align}
	Then \(\Lo^{*}_\mathrm{low}(\lopt, \rho)\) is a lower bound of \(\Lo^*(\lopt, \rho)\).
\end{corollary}

\begin{proof}
	It holds for \(\lambda_{k - 1} + 1 \leq \lopt \leq \lambda_k\) and all \(k \geq 1\) that
	\begin{align}
		\Lo^*_\mathrm{low}(\lopt, \rho) &= \frac{1}{\rho - 1} \left(\lopt - \lambda_{0} - \frac{\ln(\lopt / \lambda_0)}{\ln\left(\rho\right)} - 1\right)
		< \frac{1}{\rho - 1} \left(\lambda_k - \lambda_{0} - \frac{\ln\left(\lambda_{k - 1} / \lambda_0\right)}{\ln\left(\rho\right)} - 1\right).
	\end{align}
	From (\ref{eq:mk}) it follows that \(k - 1 \leq \ln\left(\lambda_{k - 1} / \lambda_0\right) / \ln\left(\rho\right)\) for all \(k \geq 1\). Using this and Theorem \ref{th:m2}, one gets for \(\lambda_{k - 1} + 1 \leq \lopt \leq \lambda_k\) and \(k \geq 1\) that
	\begin{align}
		\Lo^*_\mathrm{low}(\lopt, \rho) < \frac{1}{\rho - 1} \left(\lambda_k - \lambda_{0} - (k - 1) - 1\right) = \mathcal{F}_\mathrm{low}(k) < \mathcal{F}(k) = \Lo^*(\lambda_{k}, \rho) \leq \Lo^*(\lopt, \rho).
	\end{align}
	The last inequality follows from (\ref{eq:lossD}), i.e., \(\Lo^*(\lopt, \rho)\) decreases between \(\lambda_{k - 1} + 1\) and \(\lambda_{k}\). For \(\lopt = \lambda_0\) it holds that \(\Lo^*_\mathrm{low}(\lambda_0, \rho) = - \frac{1}{\rho - 1} < 0 = \Lo^*(\lambda_0, \rho)\) and therefore it holds that
	\begin{align}
		\Lo^*_\mathrm{low}(\lopt, \rho) < \Lo^*(\lopt, \rho)
	\end{align}
	for all \(\lopt \geq \lambda_0\).
\end{proof} 
The lower bound (\ref{eq:lLow}) is represented in Fig. \ref{fig:lossMult} by the blue dashed line. In the left figure, where \(\lambda_0\) and \(\rho\) are small, there are clear discrepancies between (\ref{eq:lLow}) and the lower corners of the loss function. In the right figure for larger values of \(\lambda_0\) and \(\rho\), the lower bound (\ref{eq:lLow}) is only slightly smaller than the lower corners of the loss function.

In Fig. \ref{fig:lossComp} the lower bound (\ref{eq:lLow}) of the \(\R^*\)-RS is compared with the lower bound (\ref{eq:lossModLow}) of the \(\R^\times\)-RS. The lower bound (\ref{eq:lLow}) of \(\Lo^*\) is slightly smaller than the lower bound (\ref{eq:lossModLow}) of \(\Lo^\times\) and is a lower bound for both loss functions. The lower bound (\ref{eq:lossModLow}) of \(\Lo^\times\) intersects \(\Lo^*\).

\subsection{The Loss Function for the \(\R^{\#}\)-RS}
The loss function of the \(\R^{\#}\)-RS is denoted as \(\Lo^\#(\lopt, \alpha)\) and represented by the gray markers in Fig. \ref{fig:lossPow}.

\begin{figure}
	\centering
	\hspace{-15pt}
	\begin{subfigure}{0.35\textwidth}
		\includegraphics[width=1\linewidth]{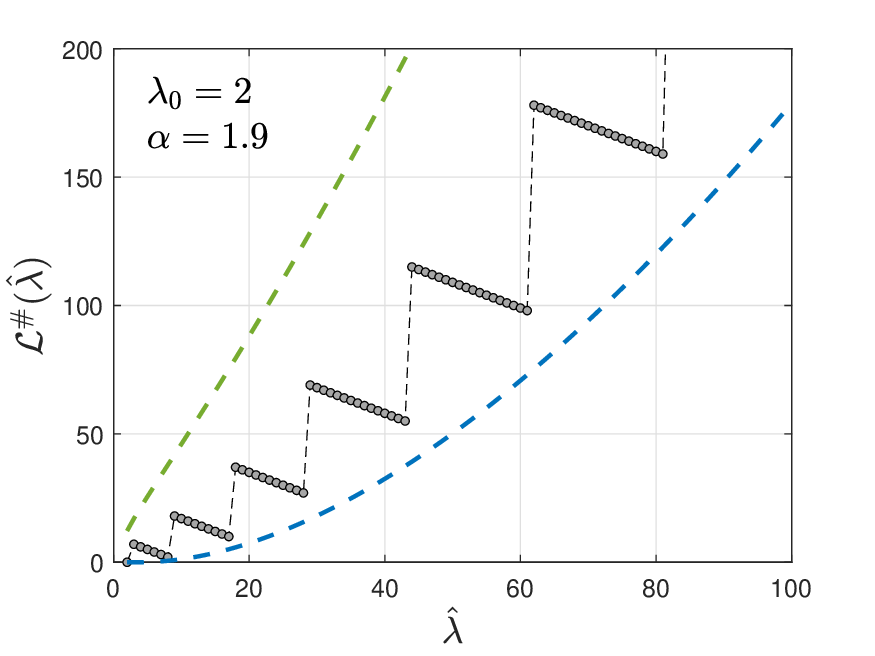}
	\end{subfigure}
	\hspace{-15pt}
	\begin{subfigure}{0.35\textwidth}
		\includegraphics[width=1\linewidth]{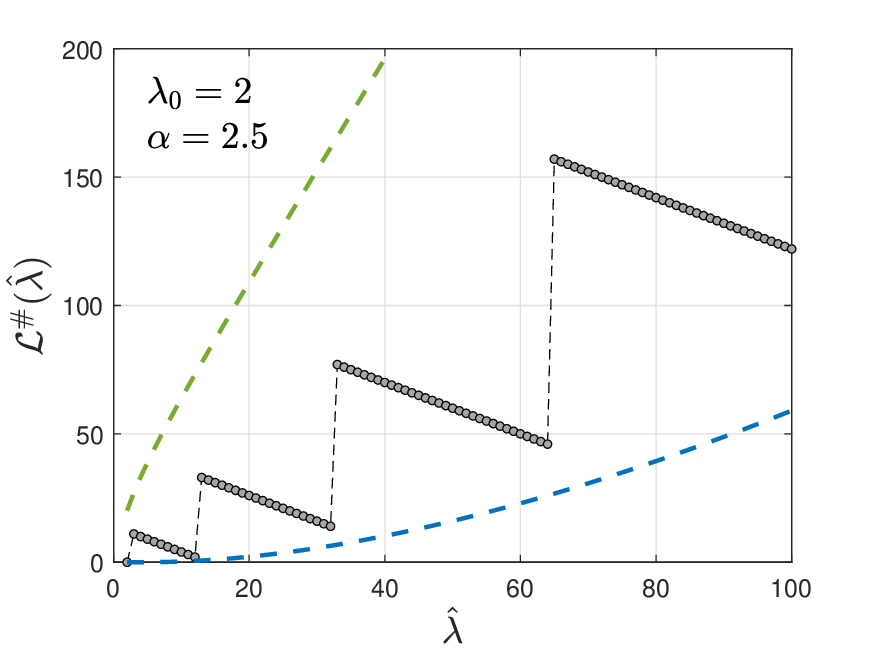}
	\end{subfigure}
	\hspace{-15pt}
	\begin{subfigure}{0.35\textwidth}
		\includegraphics[width=1\linewidth]{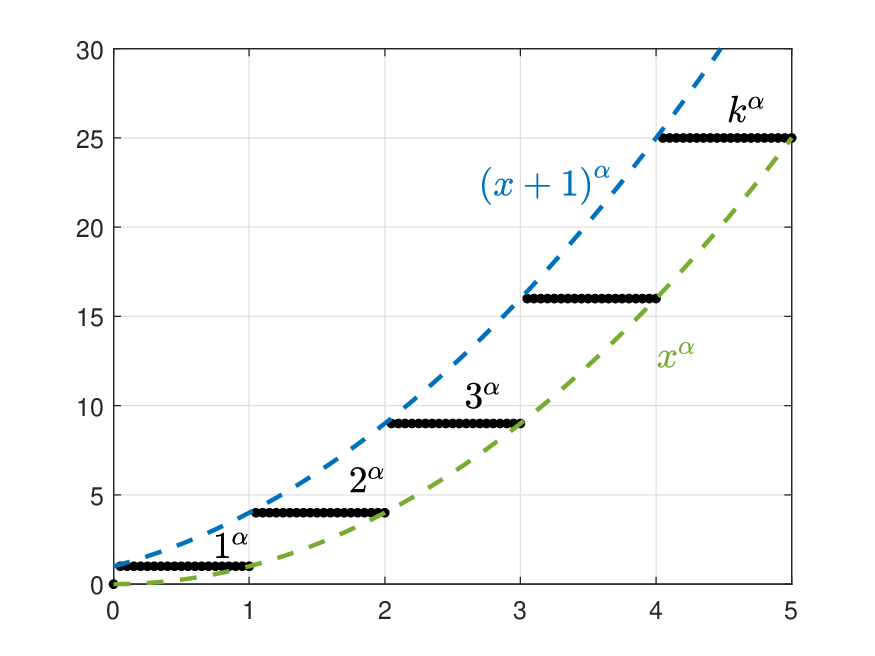}
	\end{subfigure}
	\hspace{-20pt}
	\caption{Left and middle plot: loss function (\ref{eq:lossSimple}) for \(\R^\#\) depending on \(\lopt\). Gray markers represent the numerical values of the loss function, determined with Alg. \ref{alg:lossNum}. The green dashed line shows the upper bound (\ref{eq:lossPowUp}) and the blue dashed line shows the lower bound (\ref{eq:lossPowLow}). Right plot: visualization for Eqs. (\ref{eq:intUp}) and (\ref{eq:intLow}).}
	\label{fig:lossPow}
\end{figure}

\begin{theorem} 
	Let \(\Lo^\#(\lopt, \alpha)\) be the loss function (\ref{eq:lossSimple}) of the \(\R^\#\)-RS with restart parameter \(\alpha \geq 1\) and \(\lambda_{k} = \lceil \lambda_0 (k + 1)^\alpha \rceil\). Let
	\begin{align} \label{eq:lossPowUp}
		\Lo^{\#}_\mathrm{up}\left(\lopt, \alpha\right) := \frac{\lambda_0}{\alpha + 1} \left(\sqrt[\alpha]{\frac{\lopt - 1}{\lambda_0}} + 2\right)^{\alpha + 1} + \sqrt[\alpha]{\frac{\lopt - 1}{\lambda_0}} - \frac{\lambda_{0}}{\alpha + 1}.
	\end{align}
	Then \(\Lo^{\#}_\mathrm{up}(\lopt, \alpha)\) is an upper bound of \(\Lo^\#(\lopt, \alpha)\).
\end{theorem}
\begin{proof}
	The loss function jumps at \(\lambda_{k} + 1\) and the number of restarts is \(\kopt(\lambda_k + 1) = k + 1\). It holds that
	\begin{align} \label{eq:pow0}
		\lambda_0 (k + 1)^\alpha \leq \lceil \lambda_0 (k + 1)^\alpha \rceil = \lambda_k < \lambda_0(k + 1)^\alpha + 1.
	\end{align}
	For the calculations to come the following estimation comes in handy
	\begin{align} \label{eq:intUp}
		\sum_{j = 1}^{k} j^\alpha < \int_{0}^{k}(x +1)^\alpha \mathrm{d}x = \frac{(k + 1)^{\alpha + 1}}{\alpha + 1} - \frac{1}{\alpha + 1}
	\end{align} 
	as visualized in the right plot of Fig. \ref{fig:lossPow}. Therefore, it holds for all \(k \geq 0\) that
	\begin{align} \label{eq:lo0}
		\Lo^\#(\lambda_k + 1, \alpha) &= \sum_{j = 0}^{k + 1}\lambda_j - \lambda_k - 1 = \sum_{j = 0}^{k + 1}\lceil \lambda_0 (j + 1)^\alpha\rceil - \lambda_k - 1 \nonumber
		\\
		&<  \sum_{j = 0}^{k + 1}\left(\lambda_0 (j + 1)^\alpha + 1\right) - \lambda_0(k + 1)^\alpha - 1 = \lambda_0  \sum_{j = 0}^{k + 1}(j + 1)^\alpha + k + 2 - \lambda_0(k + 1)^\alpha - 1 \nonumber
		\\
		&= \lambda_0  \sum_{i = 1}^{k + 2}i^\alpha + k + 1 - \lambda_0(k + 1)^\alpha \nonumber
		\\
		&< \lambda_0 \frac{(k + 3)^{\alpha + 1}}{\alpha + 1} - \frac{\lambda_0}{\alpha + 1} + k + 1 - \lambda_0(k + 1)^\alpha.
	\end{align}
	It follows from (\ref{eq:pow0}) that for \(\lambda_k > \lambda_0\)
	\begin{align}
		(k + 1)^\alpha \leq \frac{\lambda_{k}}{\lambda_{0}}&\Leftrightarrow 
		k \leq \sqrt[\alpha]{\frac{\lambda_k}{\lambda_0}} - 1
		\\
		(k + 1)^\alpha > \frac{\lambda_{k} - 1}{\lambda_{0}}&\Leftrightarrow 
		k > \sqrt[\alpha]{\frac{\lambda_k - 1}{\lambda_0}} - 1. \label{eq:powk}
	\end{align}
	Inserting this into (\ref{eq:lo0}), then it follows for \(k \geq 0\) and \(\lambda_k \geq 1\)
	\begin{align} \label{eq:pow1}
		\Lo^\#(\lambda_k + 1, \alpha) &< \frac{\lambda_0}{\alpha + 1} \left(\sqrt[\alpha]{\frac{\lambda_k}{\lambda_0}} + 2\right)^{\alpha + 1} - \frac{\lambda_{0}}{\alpha + 1} + \sqrt[\alpha]{\frac{\lambda_k}{\lambda_0}} -
		\lambda_0 \frac{\lambda_k - 1}{\lambda_0} \nonumber
		\\
		&= \frac{\lambda_0}{\alpha + 1} \left(\sqrt[\alpha]{\frac{\lambda_k}{\lambda_0}} + 2\right)^{\alpha + 1} + \sqrt[\alpha]{\frac{\lambda_k}{\lambda_0}} - \frac{\lambda_{0}}{\alpha + 1} - \lambda_k + 1\nonumber
		\\
		&\leq \frac{\lambda_0}{\alpha + 1} \left(\sqrt[\alpha]{\frac{\lambda_k}{\lambda_0}} + 2\right)^{\alpha + 1} + \sqrt[\alpha]{\frac{\lambda_k}{\lambda_0}} - \frac{\lambda_{0}}{\alpha + 1} - 0 = \Lo^\#_\mathrm{up}(\lambda_k + 1, \alpha).
	\end{align}
	For \(\lopt = \lambda_0\) it holds that
	\begin{align}
		\Lo^\#_\mathrm{up}(\lambda_0, \alpha) &= \frac{\lambda_0}{\alpha + 1} \left(\sqrt[\alpha]{\frac{\lambda_0 - 1}{\lambda_0}} + 2\right)^{\alpha + 1} + \sqrt[\alpha]{\frac{\lambda_0 - 1}{\lambda_0}} - \frac{\lambda_{0}}{\alpha + 1} \nonumber
		\\
		& =\frac{\lambda_0}{\alpha + 1} \left(\left(\sqrt[\alpha]{\frac{\lambda_0 - 1}{\lambda_0}} + 2\right)^{\alpha + 1} - 1\right) + \sqrt[\alpha]{\frac{\lambda_0 - 1}{\lambda_0}} \nonumber
		\\
		&\geq \frac{\lambda_0}{\alpha + 1} \left(2^{\alpha + 1} - 1\right) + \sqrt[\alpha]{\frac{\lambda_0 - 1}{\lambda_0}} > 0 = \Lo^\#(\lambda_0, \alpha).
	\end{align}
	In (\ref{eq:lossD}) it was shown that \(\Lo^\#(\lopt, \alpha)\) decreases between \(\lambda_{k} + 1\) and \(\lambda_{k + 1}\). Because \(\Lo^\#_\mathrm{up}\) is an increasing function of \(\lopt\) it holds for all \(\lopt \geq \lambda_{0}\) that
	\begin{align}
		\Lo^\#_\mathrm{up}(\lopt, \alpha) > \Lo^\#(\lopt, \alpha).
	\end{align}
\end{proof}
\noindent
The upper bound (\ref{eq:lossPowUp}) is represented by the green dashed line in the left and middle plot of Fig.~\ref{fig:lossPow}.

\begin{theorem} 
	Let \(\Lo^\#(\lopt, \alpha)\) be the loss function (\ref{eq:lossSimple}) of the \(\R^\#\)-RS with restart parameter \(\alpha \geq 1\) and \(\lambda_{k} = \lceil \lambda_0 (k + 1)^\alpha \rceil\). Let
	\begin{align} \label{eq:lossPowLow}
		\Lo^{\#}_\mathrm{low}\left(\lopt, \alpha\right) := \begin{cases}
			\frac{\lambda_0}{\alpha + 1} \left(\sqrt[\alpha]{\frac{\lopt - 1}{\lambda_0}} - 1\right)^{\alpha + 1} \quad &\textnormal{for } \lopt > \lambda_0
			\\
			0 &\textnormal{for } \lopt = \lambda_0
		\end{cases}.
	\end{align}
	Then \(\Lo^{\#}_\mathrm{low}(\lopt, \alpha)\) is a lower bound of \(\Lo^\#(\lopt, \alpha)\).
\end{theorem}
\begin{proof}
	It holds that
	\begin{align} \label{eq:intLow}
		\sum_{j = 1}^{k} j^\alpha > \int_{0}^{k}x^\alpha \mathrm{d}x = \frac{k^{\alpha + 1}}{\alpha + 1}
	\end{align} 
	as visualized in the right plot of Fig. \ref{fig:lossPow}. It follows with (\ref{eq:pow0}) and (\ref{eq:powk})
	\begin{align} \label{eq:pow2}
		\Lo^\#(\lambda_k, \alpha) &= \sum_{j = 0}^{k}\lambda_j - \lambda_k = \sum_{j = 0}^{k - 1}\lambda_j = \sum_{j = 0}^{k - 1}\lceil \lambda_0 (j + 1)^\alpha\rceil \geq \lambda_0 \sum_{j = 0}^{k - 1}(j + 1)^\alpha = \lambda_0 \sum_{i = 1}^{k}i^\alpha \nonumber
		\\
		&> \lambda_0 \frac{k^{\alpha + 1}}{\alpha + 1} > \frac{\lambda_0}{\alpha + 1} \left(\sqrt[\alpha]{\frac{\lambda_k - 1}{\lambda_0}} - 1\right)^{\alpha+1} = \Lo^\#_\mathrm{low}(\lambda_k, \alpha)
	\end{align}
	for all \(k \geq 1\)\footnote{In the case of \(\lopt = \lambda_0\), the power term in (\ref{eq:pow2}) may be imaginary.}. In (\ref{eq:lossD}) it was shown that \(\Lo^\#(\lopt, \alpha)\) decreases between \(\lambda_{k - 1} + 1\) and \(\lambda_{k}\). Therefore, it follows for \(\lambda_{k - 1}  + 1 \leq \lopt < \lambda_{k}\) by using (\ref{eq:pow2}) that
	\begin{align}
		\Lo^\#(\lopt, \alpha) > \Lo^\#(\lambda_k, \alpha) > \frac{\lambda_0}{\alpha + 1} \left(\sqrt[\alpha]{\frac{\lambda_k - 1}{\lambda_0}} - 1\right)^{\alpha+1} > \frac{\lambda_0}{\alpha + 1} \left(\sqrt[\alpha]{\frac{\lopt - 1}{\lambda_0}} - 1\right)^{\alpha+1} = \Lo_\mathrm{low}^\#(\lopt, \alpha)
	\end{align}
	which holds for all \(k \geq 1\). Because \(\Lo^\#(\lambda_0, \alpha) = 0 = \Lo_\mathrm{low}^\#(\lambda_0, \alpha)\) it holds for all \(\lopt \geq \lambda_0\) that
	\begin{align}
		\Lo^\#(\lopt, \alpha) \geq \Lo_\mathrm{low}^\#(\lopt, \alpha).
	\end{align}
\end{proof}

The lower bound (\ref{eq:lossPowLow}) is represented by the blue dashed line in the left and middle plot of Fig.~\ref{fig:lossPow}. The figure shows that (\ref{eq:lossPowUp}) and (\ref{eq:lossPowLow}) provide only very rough bounds. The reason for this is that the estimates (\ref{eq:intUp}) and (\ref{eq:intLow}) give much larger or smaller values. Nevertheless, the subsequent section will demonstrate that the \(\R^\#\)-RS fails to satisfy a fundamental criterion of optimality. Consequently, there is no necessity for a more precise estimation of the power sums in (\ref{eq:intUp}) and (\ref{eq:intLow}).

\section{The Relative Loss Function} \label{sec:rloss}

For a fixed value of the restart parameter, the loss is unbounded for \(\lopt\) and tends to infinity. This holds for all strategy types, as evidenced by Figs. \ref{fig:lossPlus}, \ref{fig:lossTimes}, \ref{fig:lossMult}, and \ref{fig:lossPow}, presented in the previous section. To further characterize the restart effort, it is useful to introduce the relative loss. It measures the loss w.r.t. the minimal \(\lambda = \hat{\lambda}\) needed to complete the algorithm successfully. The relative loss is defined by \footnote{\(\rho\) is used as a substitute to show the dependence of the restart parameter. It can be replaced by \(\nu\) and \(\alpha\), respectively.}
\begin{align} \label{eq:rl}
	\ell(\lopt, \rho) := \frac{\Lo(\lopt, \rho)}{\lopt}.
\end{align}
If \(\ell_\mathrm{up}(\lopt, \rho)\) is an upper bound of the relative loss function, the value
\begin{align} \label{rlAs}
	\overline{\ell}_\mathrm{up}(\rho) := \lim_{\lopt \rightarrow \infty} \ell_\mathrm{up}\left(\lopt, \rho\right)
\end{align}
is called the \emph{asymptotic upper bound} of the relative loss function. The question to be addressed here is whether a finite asymptotic upper bound exists for a given strategy type. If a finite asymptotic upper bound \(\overline{\ell}_\mathrm{up} (\rho)\) exists for a given restart parameter then the RS is termed to be \emph{bounded}. Conversely, if no finite upper bound exists, the strategy is called \emph{unbounded}.

The magnitude of the loss depends strongly on the restart parameters. It can therefore be assumed that the asymptotic upper bound also depends on the choice of the restart parameter. In the case of a bounded RS, it is possible to search for the optimal restart parameter, that is, the one which yields the minimal asymptotic upper bound. A bounded RS is referred to as an \emph{asymptotically optimal} RS if there exists an restart parameter for which the asymptotic upper bound is minimal. This value is called the \emph{optimal choice} for the restart parameter, which is independent of \(\lopt\). The optimal choice of the restart parameter is also denoted with \(\hat{\rho}\).

Similarly, one can define an \emph{asymptotic lower bound} and an optimal choice of the restart parameter with respect to the lower bound. If there exist an infinite asymptotic lower bound, the RS is called \emph{strictly unbounded}. It is evident that strict unboundedness sufficiently implies unboundedness. Nevertheless, the converse is not necessarily true.

\subsection{Relative Loss Function and Optimal \(\nu\) of the \(\R^+\)-RS}

Assuming that \(\lopt\) is known, then the loss functions are functions that depend on the restart parameter. The left plot of Fig. \ref{fig:rlPlus} represents this for the \(\R^+\)-RS. The blue line represents the lower bound (\ref{eq:lossPLow}) and the green line represents the upper bound (\ref{eq:lossPUp}). The question is whether there is an optimal choice for \(\nu\) that minimizes the loss. The derivative of the lower bound (\ref{eq:lossPLow}) is
\begin{align}
	\frac{\mathrm{d}}{\mathrm{d}\nu} \Lo^+_\mathrm{low}(\nu) = -\frac{(\lopt + \lambda_0)(\lopt - \lambda_0)}{2\nu^2},
\end{align}
which is non-zero for all \(\nu\) and for \(\lopt > \lambda_0\). There is no optimal value of \(\nu\) that minimizes the lower bound of the \(\R^+\) loss function. This is also visible in the left plot of Fig. \ref{fig:rlPlus}.
The derivative of the upper bound (\ref{eq:lossPUp}) is
\begin{align}
	\frac{\mathrm{d}}{\mathrm{d}\nu} \Lo^+_\mathrm{up}(\nu) = -\frac{(\lopt - \lambda_0 - 1) (\lopt + \lambda_0 - 1)}{2\nu^2} + 1 = -\frac{\lopt^2 - 2 \lopt- \lambda_0^2 + 1 }{2\nu^2} + 1 = -\frac{(\lopt - 1)^2 - \lambda_0^2}{2\nu^2} + 1.
\end{align}
Setting the derivative to zero yields
\begin{align} \label{eq:nuOpt}
	\nopt = \sqrt{\halb ((\lopt - 1)^2 - \lambda_0^2)},
\end{align}
which minimizes the upper bound of the loss function. This is evident in the left plot of Fig. \ref{fig:rlPlus}, where the minimum of the upper bound occurs at \(\nu \approx 70\). It is important to note that the value of \(\nopt\) (\ref{eq:nuOpt}) strongly depends on \(\lopt\), which is generally unknown. As a result, there does not exist an optimal \(\hat{\nu}\) independent of the unknown \(\hat{\lambda}\). Therefore, \(\R^+\) cannot be an asymptotically optimal restart strategy. Moreover, \(\R^+\) is strictly unbounded, which is shown in the following theorem.

\begin{figure}
	\centering
	\begin{subfigure}{.45\textwidth}
		\includegraphics[width=1\linewidth]{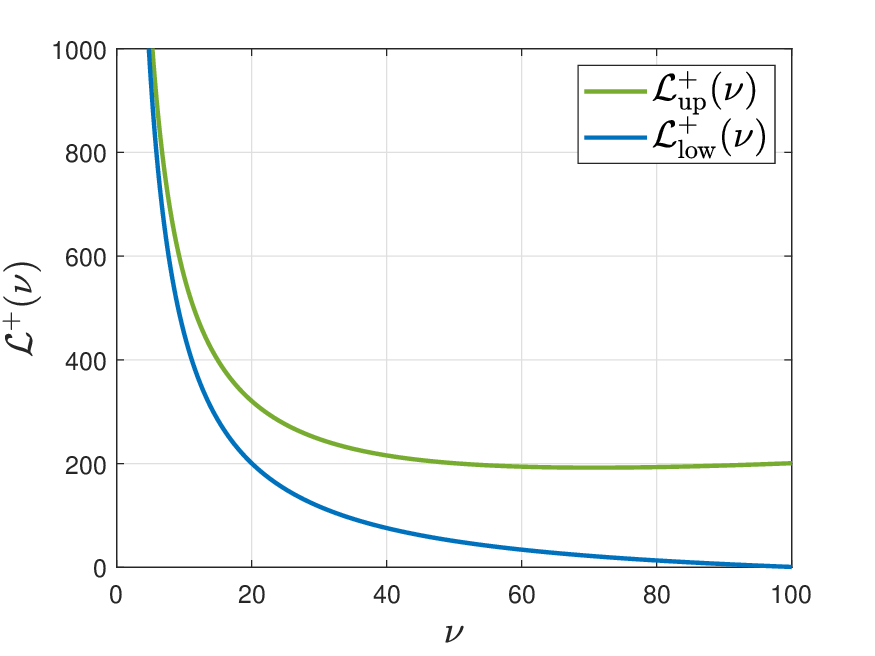}
	\end{subfigure}
	\begin{subfigure}{.45\textwidth}
		\includegraphics[width=1\linewidth]{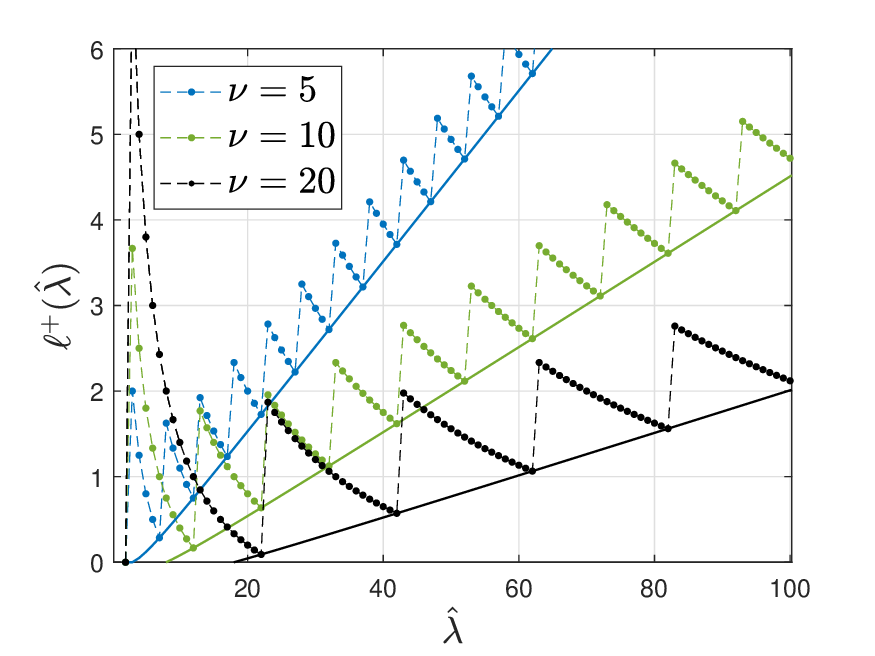}
	\end{subfigure}
	\caption{Left plot: lower bound (\ref{eq:lossPLow}) and upper bound (\ref{eq:lossPUp}) of the \(\R^+\) loss function depending on \(\nu\) for \(\lambda_0 = 2\) and \(\lopt = 100\). Right plot: markers show the relative loss function (\ref{eq:rlPlus}) for \(\lambda_0 = 2\) and solid lines represent the corresponding lower bound (\ref{eq:relP0}).}
	\label{fig:rlPlus}
\end{figure} 

\begin{theorem}
	Let 
	\begin{align} \label{eq:rlPlus}
		\ell^+(\lopt, \nu) := \frac{\Lo^+(\lopt, \nu)}{\lopt}
	\end{align}
	be the relative loss function (\ref{eq:rl}) for the \(\R^+\)-RS. Then \(\ell^+(\lopt, \nu)\) is strictly unbounded for all \(\nu \in \mathbb{N} \setminus \{0\}\).
\end{theorem}

\begin{proof}
	In Theorem (\ref{th:pLow}) it was demonstrated that
	\begin{align} 
		\Lo^+_\mathrm{low}\left(\lopt, \nu\right) =  \halb \left(\lopt - \lambda_0\right)\left(\frac{\lopt + \lambda_0}{\nu} - 1\right).
	\end{align}
	is a lower bound of the loss function of the \(\R^+\)-RS. This implies that 
	\begin{align} \label{eq:relP0}
		\ell^+_\mathrm{low}\left(\lopt, \nu\right) := \frac{\Lo^+_\mathrm{low}\left(\lopt, \nu\right)}{\lopt} =  \halb \left(1 - \frac{\lambda_0}{\lopt}\right)\left(\frac{\lopt + \lambda_0}{\nu} - 1\right) = \frac{\lopt + \lambda_0}{2\nu} - \frac{\left(\lopt + \lambda_0\right)\lambda_0}{2\lopt\nu} - \halb + \frac{\lambda_0}{2\lopt}
	\end{align}
	is a lower bound of the relative loss function (\ref{eq:rlPlus}). For \(\lopt \rightarrow \infty\) the first expression of (\ref{eq:relP0}) tends to infinity. The remaining expressions are finite. Therefore, an infinite asymptotic lower bound exists.
\end{proof}

The relative loss (\ref{eq:rlPlus}) is illustrated in the right plot of Fig. \ref{fig:rlPlus} for varying values of \(\nu\). It can be seen that for a fixed \(\nu\) the relative loss exhibits a linear trend with \(\lopt\) and approaches infinity. Furthermore, for large \(\lopt\) values the relative loss increases for smaller \(\nu\) values. The solid lines in the right plot of Fig. \ref{fig:rlPlus} show the lower bound (\ref{eq:relP0}) of the relative loss function.

\subsection{Relative Loss and Optimal \(\rho\) of the Multiplicative Strategy Types}

The relative loss (\ref{eq:rl}) of the \(\R^\times\)-RS is given by
\begin{align} \label{eq:rlMod}
	\ell^\times \left(\lopt, \rho\right) = \frac{\Lo^\times \left(\lopt, \rho\right)}{\lopt}.
\end{align}
Figure \ref{fig:rlMod} shows (\ref{eq:rlMod}) depending on \(\lopt\). In contrast to the \(\R^+\)-RS (see the right plot of Fig. \ref{fig:rlPlus}) the relative loss exhibits an upper bound. In order to derive the upper bound for the relative loss, start with the upper bound of the loss function (\ref{eq:lossModUp}) and divide it by \(\lopt\), i.e.,
\begin{align} \label{eq:rlModUp}
	\ell_\mathrm{up}^\times \left(\lopt, \rho\right) &:= \frac{\Lo_\mathrm{up}^\times \left(\lopt, \rho\right)}{\lopt} = \frac{1}{\lopt} \left( \lopt \left(\rho + \frac{1}{\rho - 1}\right) - \frac{\lambda_0}{\rho - 1} + \frac{\ln\left(\lopt / \lambda_0\right)}{\ln\left(\rho\right)}\right) \nonumber
	\\
	&= \rho + \frac{1}{\rho - 1} - \frac{\lambda_0}{\lopt(\rho - 1)} + \frac{\ln\left(\lopt / \lambda_0\right)}{\lopt\ln\left(\rho\right)}.
\end{align}
The corresponding asymptotic upper bound is
\begin{align} \label{eq:rlAsMod}
	\overline{\ell}_\mathrm{up}^\times(\rho) = \lim_{\lopt \rightarrow \infty}\ell_\mathrm{up}^\times \left(\lopt, \rho\right) = \rho + \frac{1}{\rho - 1}.
\end{align}
It can be seen that this expression is finite for all \(\rho > 1\). It can thus be concluded that the \(\R^\times\)-RS is bounded. The upper bound (\ref{eq:rlModUp}) is illustrated in Fig. \ref{fig:rlMod} by the solid lines. It can be observed that for sufficient large \(\lopt\), the upper bound provides a satisfactory approximation of the upper corners of the relative loss function.

The left plot of Fig. \ref{fig:rlMod} shows that the maximum relative loss increases with smaller \(\rho\), while the middle plot shows that the maximum relative loss increases with larger \(\rho\). Therefore, it can be assumed that there is a value of \(\rho\) where the maximal relative loss is minimal. In both figures, the smallest relative loss occurs at about \(\rho = 2\). For sufficiently large \(\lopt\), these observations are independent of \(\lambda_0\). This can be seen in the right plot of Fig. \ref{fig:rlMod}, which shows the loss functions for \(\lambda_0 = 10\). The maximum values of the asymptotic behavior are identical to those for \(\lambda_0 = 2\). These observations are confirmed by the following theorem:

\begin{figure}
	\centering
	\hspace{-15pt}
	\begin{subfigure}{.35\textwidth}
		\includegraphics[width=1\linewidth]{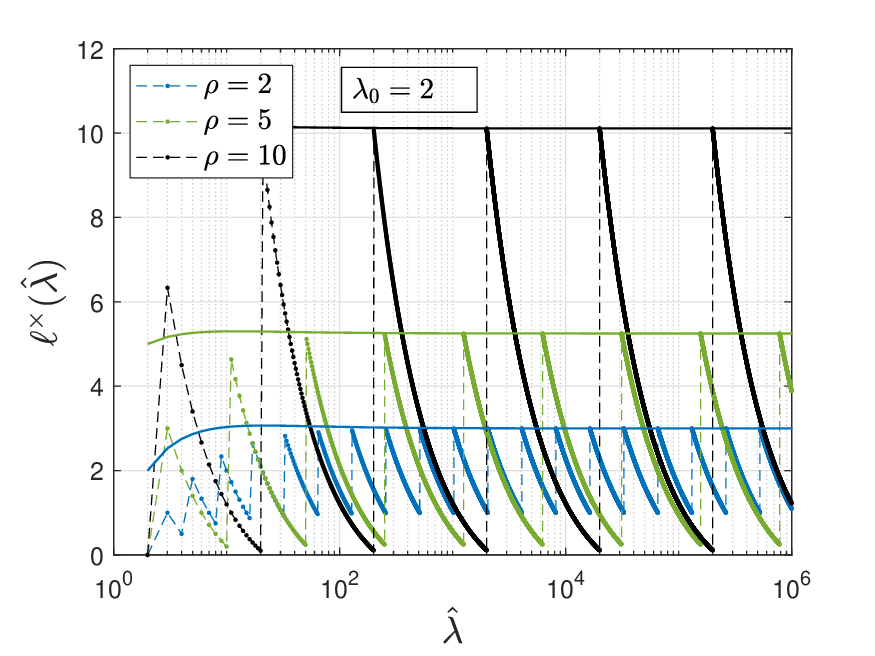}
	\end{subfigure}
	\hspace{-15pt}
	\begin{subfigure}{.35\textwidth}
		\includegraphics[width=1\linewidth]{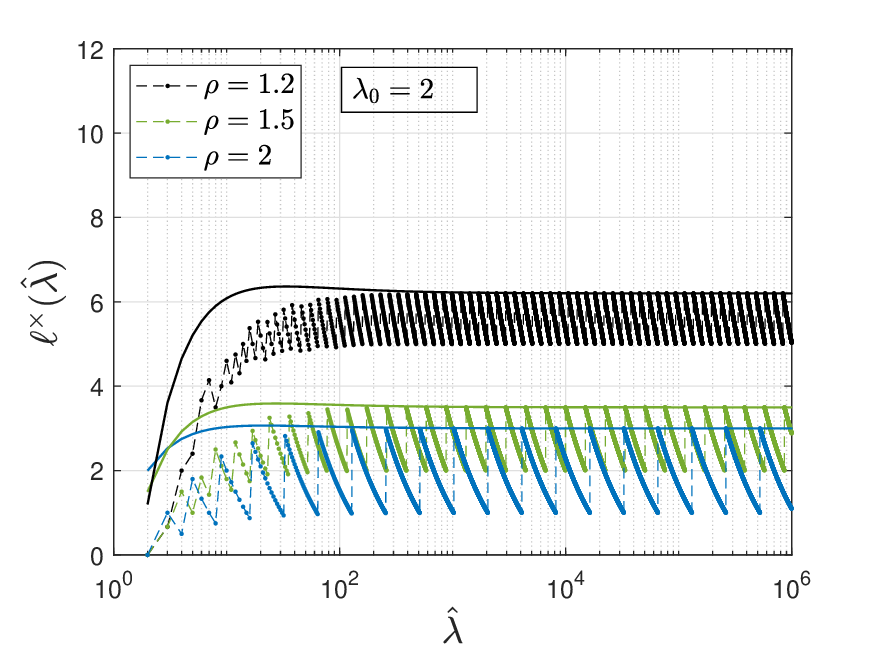}
	\end{subfigure}
	\hspace{-15pt}
	\begin{subfigure}{.35\textwidth}
		\includegraphics[width=1\linewidth]{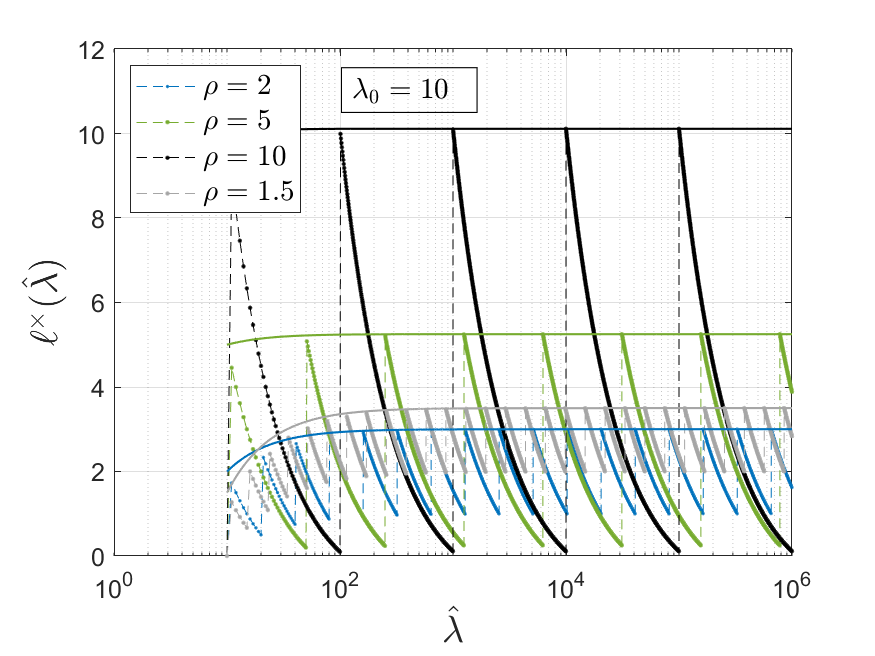}
	\end{subfigure}
	\hspace{-20pt}
	\caption{Markers with dashed lines: relative loss function (\ref{eq:rlMod}) of the \(\R^{\times}\)-RS depending on \(\lopt\). Values were determined numerically with Alg. \ref{alg:lossNum}. Solid lines: upper bound of the relative loss function (\ref{eq:rlModUp}).}
	\label{fig:rlMod}
\end{figure}

\begin{theorem} \label{th:rlMod}
	The \(\R^\times\)-RS is an asymptotic optimal RS with \(\ropt = 2\). Furthermore, it holds for the asymptotic upper bound (\ref{eq:rlAsMod}) that \(\overline{\ell}^{\times}_\mathrm{up}\left(\ropt\right) = 3.\)
\end{theorem}

\begin{proof}
	Setting the derivative of the asymptotic upper bound (\ref{eq:rlAsMod}) to zero yields
	\begin{align} 
		\frac{\mathrm{d}}{\mathrm{d}\rho} \overline{\ell}^{\times}_\mathrm{up}(\rho) = 1 - \frac{1}{\left(\rho - 1\right)^2} = \frac{\rho^2 - 2\rho}{\left(\rho - 1\right)^2} = 0 \Leftrightarrow 2 \rho = \rho^2 \Leftrightarrow \rho = 2,
	\end{align}
	thus indicating that the optimal restart parameter is \(\ropt = 2\).
	Inserting this into the asymptotic upper bound (\ref{eq:rlModUp}), it follows that
	\begin{align}
		\overline{\ell}^{\times}_\mathrm{up}(\ropt) = 2 +\frac{1}{2 - 1} = 3.
	\end{align}
\end{proof}

Figure \ref{fig:rlMod} illustrates that the asymptotic lower bound of the relative loss function is not minimal for \(\rho = 2\). Instead, it is a monotonously decreasing function. The lower bound (\ref{eq:lossModLow}) implies that
\begin{align}
	\overline{\ell}_\mathrm{low}^\times \left(\rho\right) = \lim_{\lopt \rightarrow \infty} \frac{\Lo^\times_\mathrm{low}\left(\lopt, \rho\right)}{\lopt} = \lim_{\lopt \rightarrow \infty} \frac{\lopt - \lambda_0 - 1}{\lopt \left(\rho - 1\right) } = \frac{1}{\rho - 1}.
\end{align}
is an asymptotic lower bound of the relative loss function of the \(\R^\times\)-RS. This expression is minimized when \(\rho \rightarrow \infty\). Therefore, an asymptotic lower bound exist for all \(\rho\), but there is not an optimal choice of \(\rho\) w.r.t. the lower bound.

The relative loss of the \(\R^*\)-RS is given by
\begin{align} \label{eq:rlM}
	\ell^* \left(\lopt, \rho\right) = \frac{\Lo^* \left(\lopt, \rho\right)}{\lopt}
\end{align}
and is represented by the markers with dashed lines in Fig. \ref{fig:rlM}. It is evident that for integer values of \(\rho\), the relative loss curves are identical to those of Fig. \ref{fig:rlMod}. However, even for decimal numbers, the differences between the relative loss of the \(\R^\times\)-RS and the \(\R^*\)-RS are only barely visible. This behavior is as expected, as demonstrated in Section \ref{sec:loss}, where it was shown that the upper bound (\ref{eq:lossUp}) of the \(\R^*\)-RS is identical to that of the \(\R^\times\)-RS. It can thus be concluded that the upper bound of the relative loss function and the asymptotic upper bounds are identical, i.e.,
\begin{align} \label{eq:rlAsM}
	\overline{\ell}_\mathrm{up}^*(\rho) = \rho + \frac{1}{\rho - 1}
\end{align}
is an asymptotic upper bound of the \(\R^*\)-RS. In consequence, Theorem \ref{th:rlMod} holds also for the \(\R^*\)-RS.

\begin{figure}
	\centering
	\hspace{-15pt}
	\begin{subfigure}{.35\textwidth}
		\includegraphics[width=1\linewidth]{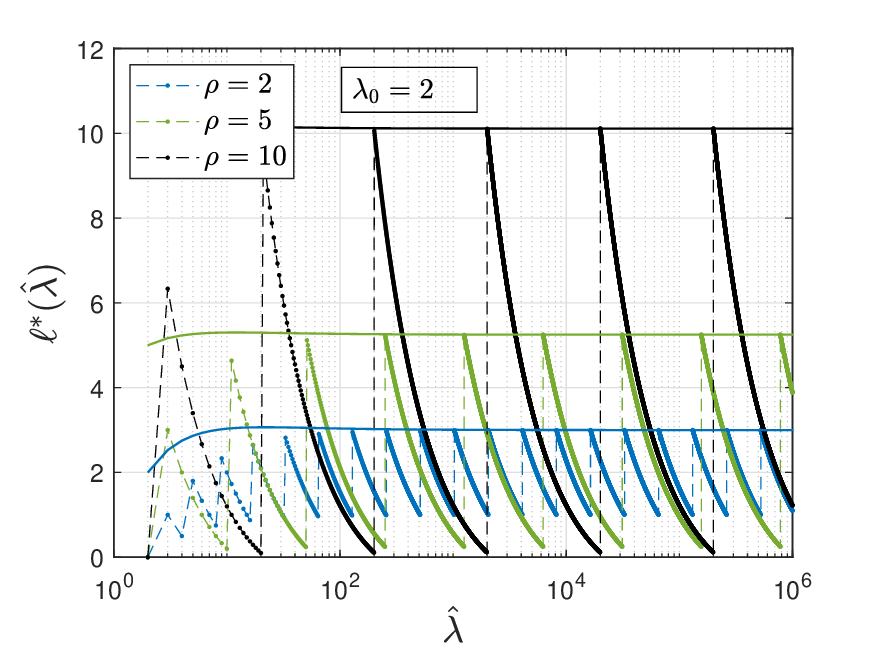}
	\end{subfigure}
	\hspace{-15pt}
	\begin{subfigure}{.35\textwidth}
		\includegraphics[width=1\linewidth]{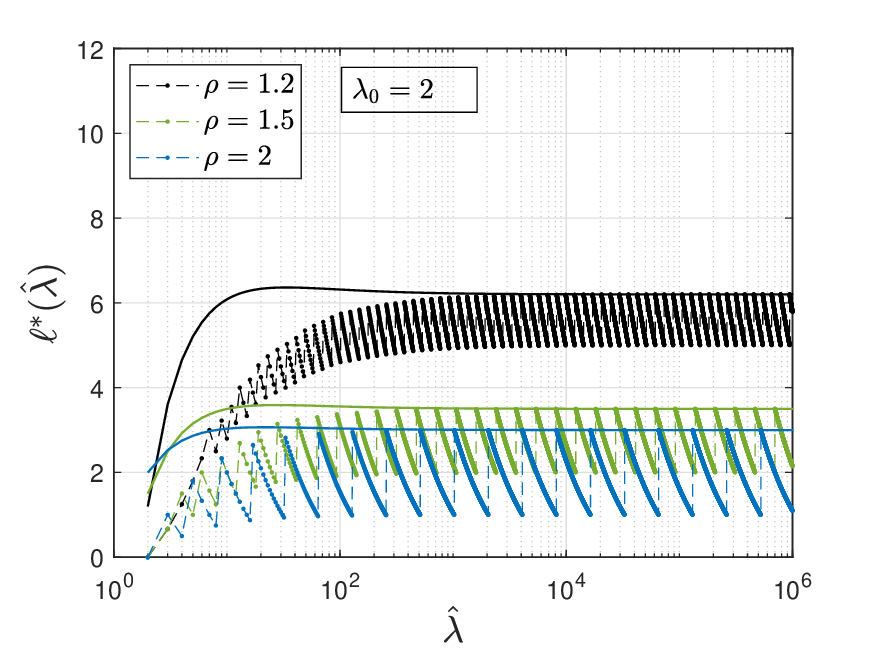}
	\end{subfigure}
	\hspace{-15pt}
	\begin{subfigure}{.35\textwidth}
		\includegraphics[width=1\linewidth]{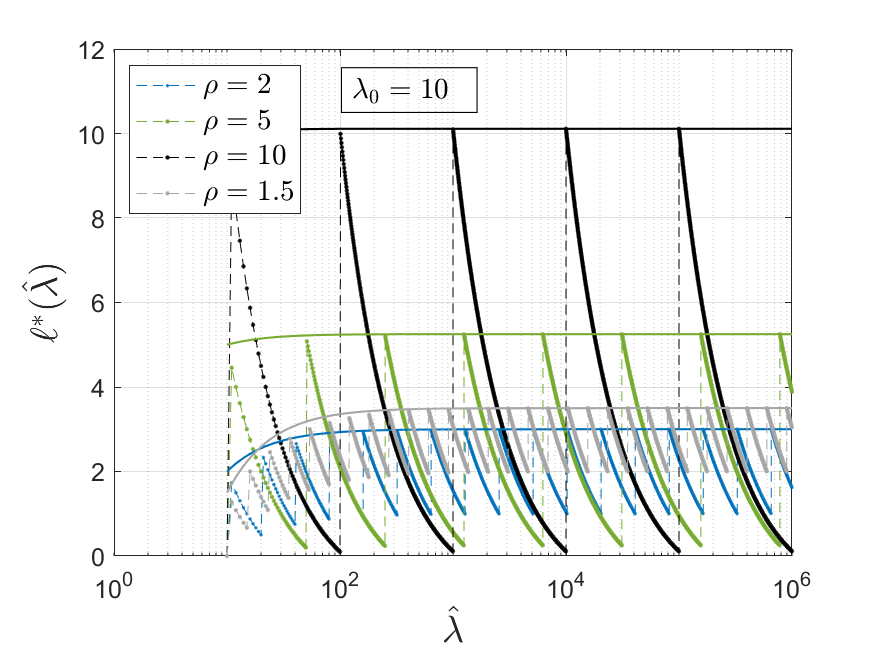}
	\end{subfigure}
	\hspace{-20pt}
	\caption{Markers with dashed lines: relative loss function (\ref{eq:rlM}) of the \(\R^*\)-RS depending on \(\lopt\). Values were determined numerically with Alg. \ref{alg:lossNum}. Solid lines: upper bound of the relative loss function.}
	\label{fig:rlM}
\end{figure}

\begin{theorem} \label{th:rlM}
	The \(\R^*\)-RS is an asymptotic optimal RS with \(\ropt = 2\). Furthermore, it holds for the asymptotic upper bound (\ref{eq:rlAsM}) that \(\overline{\ell}^*_\mathrm{up}\left(\ropt\right) = 3.\)
\end{theorem}

\begin{proof}
	Identical to the proof of Theorem \ref{th:rlMod}.
\end{proof}

The lower bounds of the two strategy types, however, are slightly different. By using (\ref{eq:lLow}) the asymptotic lower bound is
\begin{align}
	\overline{\ell}_\mathrm{low}^* \left(\rho\right) = \lim_{\lopt \rightarrow \infty} \frac{\Lo^*_\mathrm{low}\left(\lopt, \rho\right)}{\lopt} = \lim_{\lopt \rightarrow \infty}\frac{1}{\rho - 1} \frac{1}{\lopt}\left(\lopt - \lambda_{0} - \frac{\ln\left(\lopt / \lambda_0\right)}{\ln\left(\rho\right)} - 1\right) = \frac{1}{\rho - 1}.
\end{align}
Hence, the asymptotic lower bounds of the \(\R^*\)-RS and the \(\R^\times\)-RS are identical. Also for the \(\R^*\)-RS, there is no optimal choice of \(\rho\) w.r.t. the lower bound.

\subsection{Relative Loss of the \(\R^\#\)-RS}

The relative loss of the \(\R^\#\)-RS is defined by
\begin{align} \label{eq:rlPow}
	\ell^\# \left(\lopt, \alpha\right) := \frac{\Lo^\# \left(\lopt, \alpha\right)}{\lopt},
\end{align}
and is represented in Fig. \ref{fig:rlPow} for different values of \(\alpha\) and \(\lambda_0\). For sufficiently large \(\lopt\), the relative loss is observed to be smaller for larger values of \(\alpha\) and for larger values of \(\lambda_0\). It can be observed that for a fixed \(\alpha\), the relative loss tends to infinity with \(\lopt\). This leads to the hypothesis that \(\R^\#\) is a strictly unbounded RS.

\begin{figure}
	\centering
	\hspace{-15pt}
	\begin{subfigure}{.45\textwidth}
		\includegraphics[width=1\linewidth]{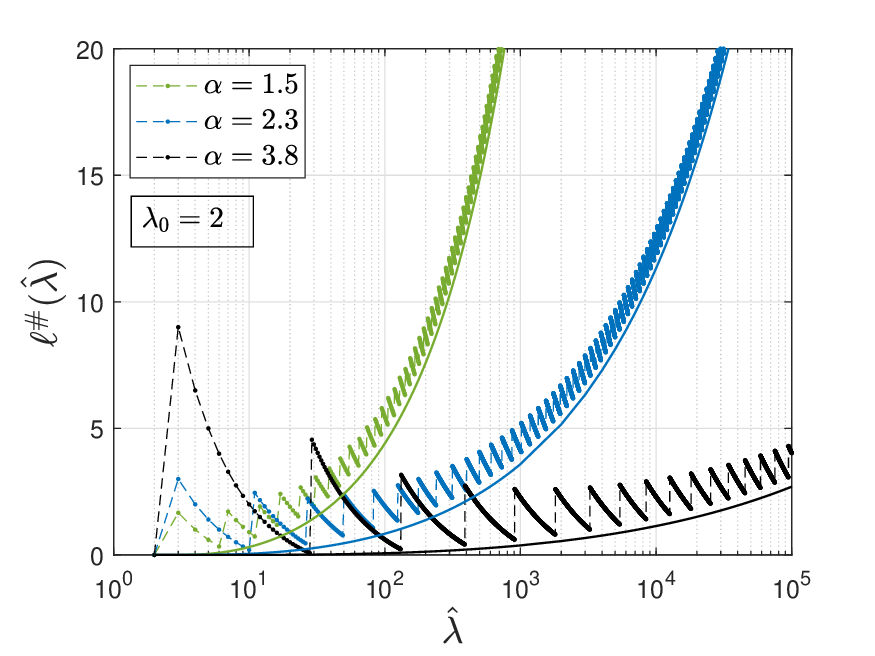}
	\end{subfigure}
	\begin{subfigure}{.45\textwidth}
		\includegraphics[width=1\linewidth]{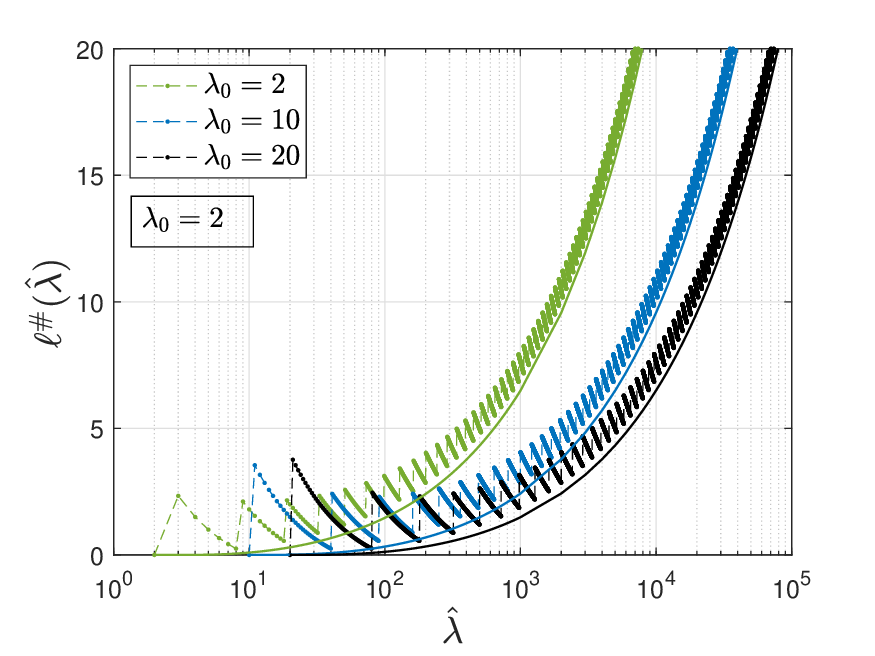}
	\end{subfigure}
	\caption{Markers with dashed lines: relative loss function (\ref{eq:rlPow}) of the \(\R^\#\)-RS depending on \(\lopt\). Values were determined numerically with Alg. \ref{alg:lossNum}. Solid lines: lower bound of the relative loss function (\ref{eq:rlPowLow}).}
	\label{fig:rlPow}
\end{figure} 

\begin{theorem}
	The \(\R^\#\) is a strictly unbounded strategy type.
\end{theorem}

\begin{proof}
	By using the lower bound (\ref{eq:lossPowLow}) of the loss function, it follows that 
	\begin{align} \label{eq:rlPowLow}
		\ell^\#_\mathrm{low}\left(\lopt, \alpha\right) &= \frac{\Lo^{\#}_\mathrm{low}\left(\lopt, \alpha\right)}{\lopt} = \frac{1}{\lopt}\frac{\lambda_0}{\alpha + 1} \left(\sqrt[\alpha]{\frac{\lopt - 1}{\lambda_0}} - 1\right)^{\alpha + 1} = \frac{1}{\lopt}\frac{\lambda_0}{\alpha + 1} \left(\sqrt[\alpha]{\frac{\lopt - 1}{\lambda_0}} \frac{\sqrt[\alpha]{\frac{\lopt - 1}{\lambda_0}} - 1}{\sqrt[\alpha]{\frac{\lopt - 1}{\lambda_0}}}\right)^{\alpha + 1} \nonumber
		\\
		&= \frac{\lambda_0}{\alpha + 1} \frac{1}{\lopt} \sqrt[\alpha]{\frac{\lopt - 1}{\lambda_0}}^{\alpha + 1} \left(1 - \sqrt[\alpha]{\frac{\lambda_0}{\lopt - 1}} \right)^{\alpha + 1}
	\end{align}
	is a lower bound of the relative loss function (\ref{eq:rlPow}) for \(\lopt > \lambda_0\). For \(\lopt \rightarrow \infty\) the expression within the parentheses approaches 1, while the first terms diverges. Consequently, the asymptotic lower bound is infinite indicating that \(\R^\#\) is a strictly unbounded RS.
\end{proof}

The solid lines in Fig. \ref{fig:rlPow} illustrate the lower bound (\ref{eq:rlPowLow}) of the relative loss function.

\section{Conclusion and Outlook} \label{sec:end}

This work examined restart strategies (RS) for algorithms that rely on an algorithmic parameter, denoted with \(\lambda\), to achieve success. The optimal choice of this algorithmic parameter for each restart is a question that has not been investigated so far. To estimate and compare different restart strategies, the set of all restart strategies was divided into parameter-dependent subsets, which were designated as strategy types. The objective was to evaluate the impact of varying restart strategies and to estimate the influence of the restart parameter. For this purpose, the loss function was introduced, which measures the number of function evaluations of an RS compared to the number of function evaluations of the optimal strategy. Due to the complexity of the loss function, upper and lower bounds of the loss function were derived for each strategy type under consideration. These bounds have been expressed as a function of the optimal \(\lambda\) and the restart parameter.

To further examine the restart effort, the relative loss has been introduced as a measure of the loss relative to the optimal \(\lambda\). One requisite for an appropriate RS is that there exists a finite upper bound for the relative loss function. Strategy types that satisfy this criterion are called bounded. For strategy types whose relative loss functions are upper bounded, it is possible to minimize the upper bound according to the parameter of the strategy type. The analyses in this paper have shown that the strategy types \(\R^+\) and \(\R^\#\) are not well-suited as restart strategies. In the case of the strategy type \(\R^+\), where the same amount \(\nu\) is added after each restart, it has been demonstrated that the relative loss function is unbounded. This is also the case for \(\R^\#\). For this strategy type the parameters are determined according to a power law. This does, however, not exclude that there are problem instances where these strategies might excel. 

In the case of the multiplicative strategy type \(\R^*\)-RS, where the algorithmic parameter is multiplied by an restart parameter \(\rho\) for each restart, it was demonstrated that the relative loss function is bounded. In examining this strategy type, it was demonstrated that there exists an optimal choice of \(\rho\), which minimizes the asymptotic upper bound of the relative loss function. This value was found to be \(\ropt = 2\). It is independent of the start value \(\lambda_0\). Furthermore, it was shown that there is no value of \(\rho\) that minimizes the asymptotic lower bound of the relative loss function. 
 
The same results can be derived for the \(R^\times\)-RS, which is also a multiplicative strategy type. In contrast to the \(R^*\)-RS, where the values of \(\lambda_{k}\) are the sum of the previous rounded values, for the \(R^\times\)-RS the values for \(\lambda_k\) are rounded up only once at the end. Using \(\rho = \ropt\), the maximum relative loss w.r.t. \(\lopt\) is 3. This result is remarkable indicating that even in the worst case the performance of the strategies degrades by only a (constant) factor of three.

In this work, the worst case scenario has been investigated. Alternatively, some kind of amortized analysis seems to be possible, i.e., the average relative loss can also be considered. The restart parameter that minimizes the average relative loss must be necessarily larger than \(\ropt\). This is a topic that is currently under investigation. 

While it has been demonstrated in this paper that the multiplicative strategy types are preferable to the other strategy types under consideration it is currently unclear whether there exist other strategy types whose relative loss is bounded as it is the case for the multiplicative strategy types. This is a topic for future research.

\section*{Acknowledgements}
This work was supported by the Austrian Science Fund (FWF) under grant P33702-N.

\end{document}